\documentclass[12pt,a4paper,final]{amsart}

\usepackage{tikz}
\usetikzlibrary{decorations.pathmorphing}
\usetikzlibrary{decorations.pathreplacing,calc}

\usepackage[curve]{xypic}
\usepackage{graphicx}
\usepackage{psfrag}
\usepackage{placeins}

\usepackage{mdwlist}
\usepackage{paralist}

\usepackage{color}
\usepackage{amssymb}
\usepackage{latexsym}
\usepackage{amsmath}
\usepackage{mathrsfs}

\usepackage{fullpage}

\newcommand{\dotspace}{\hspace{-0.25em}}
\newcommand{\blackdot}{     \raisebox{-0.3ex}{\includegraphics[height=1em]{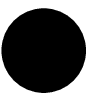}}}
\newcommand{\greydot}{      \raisebox{-0.3ex}{\includegraphics[height=1em]{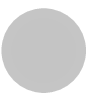}}}
\newcommand{\darkgreydot}{  \raisebox{-0.3ex}{\includegraphics[height=1em]{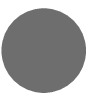}}}
\newcommand{\backgrounddot}{\raisebox{-0.3ex}{\includegraphics[height=1em]{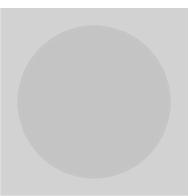}}}
\newcommand{\rimdot}{       \raisebox{-0.3ex}{\includegraphics[height=1em]{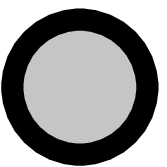}}}
\newcommand{\boxeddot}{     \raisebox{-0.3ex}{\includegraphics[height=1em]{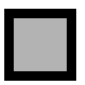}}}
\newcommand{\greyboxdot}{   \raisebox{-0.3ex}{\includegraphics[height=1em]{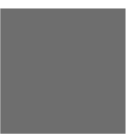}}}

\theoremstyle{plain}
\newtheorem{theorem}{Theorem}[section]

\newtheorem{lemma}[theorem]{Lemma}
\newtheorem{corollary}[theorem]{Corollary}
\newtheorem{proposition}[theorem]{Proposition}
\newtheorem{thmdefn}[theorem]{Theorem-Definition}
\newtheorem{introtheorem}{Theorem}

\theoremstyle{definition}
\newtheorem{remark}[theorem]{Remark}
\newtheorem{example}[theorem]{Example}
\newtheorem*{naive-algorithm}{Na\"ive algorithm}
\newtheorem*{refined-algorithm}{Refined algorithm}
\newtheorem*{definition}{Definition}

\newtheorem{setup}[theorem]{Setup}

\newcommand{\Db}{\sD^b}

\newcommand{\IC}{{\mathbb{C}}}

\newcommand{\IP}{{\mathbb{P}}}

\newcommand{\IZ}{{\mathbb{Z}}}

\newcommand{\IN}{{\mathbb{N}}}
\newcommand{\IX}{{\mathbb{X}}}

\newcommand{\tuberk}{\varrho}

\newcommand{\kk}{{\mathbf{k}}}

\DeclareMathOperator{\coker}{\mathsf{coker}}
\DeclareMathOperator{\kernel}{\mathsf{ker}}
\DeclareMathOperator{\Cone}{\mathsf{cone}}
\DeclareMathOperator{\Susp}{\mathsf{Susp}}

\DeclareMathOperator{\Hom}{\mathsf{Hom}}
\DeclareMathOperator{\Aut}{\mathsf{Aut}}

\DeclareMathOperator{\Ext}{\mathsf{Ext}}
\DeclareMathOperator{\coh}{\mathsf{coh}}

\DeclareMathOperator{\ind}{\mathsf{ind}\,}
\renewcommand{\mod}[1]{\mathsf{mod}(#1)}
\DeclareMathOperator{\add}{\mathsf{add}}

\newcommand{\len}[1]{\ell(#1)}

\newcommand{\lCeil}{\Big\lceil}
\newcommand{\rCeil}{\Big\rceil}

\newcommand{\card}[1]{\lvert #1\rvert}
\newcommand{\ceil}[1]{\lceil #1\rceil}
\newcommand{\floor}[1]{\lfloor #1\rfloor}

\renewcommand{\setminus}{\smallsetminus}

\newcommand{\sC}{\mathsf{C}}
\newcommand{\sD}{\mathsf{D}}

\newcommand{\sH}{\mathsf{H}}

\newcommand{\sR}{\mathsf{R}}

\newcommand{\sT}{\mathsf{T}}

\newcommand{\sW}{\mathsf{W}}
\newcommand{\sX}{\mathsf{X}}
\newcommand{\sY}{\mathsf{Y}}

\DeclareMathAlphabet{\mathpzc}{OT1}{pzc}{m}{it}

\newcommand{\cC}{\mathscr{C}}

\newcommand{\cI}{\mathscr{I}}

\newcommand{\cN}{\mathscr{N}}

\newcommand{\cP}{\mathscr{P}}

\renewcommand{\cR}{\mathscr{R}}

\newcommand{\cT}{\mathscr{T}}

\newcommand{\naiveX}[2]{\mathsf{N}\mathsf{X}(#1)_{#2}}    
\newcommand{\naiveY}[2]{\mathsf{N}\mathsf{Y}(#1)_{#2}}    
\newcommand{\refinedX}[2]{\mathsf{R}\mathsf{X}(#1)_{#2}}  
\newcommand{\refinedY}[2]{\mathsf{R}\mathsf{Y}(#1)_{#2}}  
\newcommand{\waste}[2]{\mathsf{W}(#1)_{#2}}               

\newcommand{\sod}[1]{{\langle #1\rangle}}  

\newcommand{\isom}{ \text{{\hspace{0.48em}\raisebox{0.8ex}{${\scriptscriptstyle\sim}$}}}
                    \hspace{-0.65em}{\rightarrow}\hspace{0.3em}} 

\newcommand{\arrd}{ \ar@{-}[r] \ar@{=}[d] }

\newcommand{\too}{\longrightarrow}

\newcommand{\rightlabel}[1]{\stackrel{#1}{\longrightarrow}}
\newcommand{\rightiso}{\rightlabel{\sim}}
\newcommand{\xxrightarrow}[1]{\xrightarrow{\raisebox{-0.1ex}{\ensuremath{{\scriptscriptstyle #1}}}}} 

\newcommand{\tri}[3]{#1\rightarrow #2\rightarrow #3\rightarrow \Sigma #1}

\newcommand{\hocolim}{\underrightarrow{\mathsf{holim}}\,}
\newcommand{\colim}{\underrightarrow{\mathsf{lim}}\,}

\newcommand{\bib}[6]{{\bibitem[#1]{#2} #3: {\emph{#4},} #5\texttt{#6}.}}

\usepackage[notref, notcite]{showkeys}

\begin{document}

\title[Averaging t-structures]{Averaging t-structures and extension closure of aisles}

\author{Nathan Broomhead}

\author{David Pauksztello}

\author{David Ploog}

\begin{abstract}
We ask when a finite set of t-structures in a triangulated category can be `averaged' into one t-structure or, equivalently, when the extension closure of a finite set of aisles is again an aisle. 
There is a straightforward, positive answer for a finite set of compactly generated t-structures in a big triangulated category. For piecewise tame hereditary categories, we give a criterion for when averaging is possible, and an algorithm that computes truncation triangles in this case. A finite group action on a triangulated category gives a natural way of producing a finite set of t-structures out of a given one. If averaging is possible, there is an induced t-structure on the equivariant triangulated category.
\end{abstract}

\maketitle

{\small
\setcounter{tocdepth}{1}
\tableofcontents
}

\addtocontents{toc}{\protect{\setcounter{tocdepth}{-1}}}  
\section*{Introduction} 
\addtocontents{toc}{\protect{\setcounter{tocdepth}{1}}}   

\noindent
The notion of a t-structure was introduced by Beilinson, Bernstein and Deligne in the seminal paper \cite{BBD} in order to construct perverse sheaves. A t-structure in a triangulated category $\sT$ is a pair of full subcategories $(\sX,\sY)$ satisfying certain orthogonality and generating conditions. In terminology due to Keller and Vossieck in \cite{Keller-Vossieck}, the full subcategory $\sX$ is called the `aisle' of the t-structure and $\sY$ is called the `co-aisle'.

T-structures are highly useful, and have since become a major tool in understanding the structure of triangulated categories. They can be thought of as the triangulated analogue of torsion pairs introduced in the abelian setting by Dickson in \cite{Dickson}, which abstracts the well-known decomposition of abelian groups into torsion and torsion-free abelian groups. The main applications of t-structures are the recovery of abelian categories from triangulated categories and the construction of (co)homological functors. 

Our motivation for the construction presented in this paper is the following: One can consider abstract group actions on triangulated categories, e.g.\ actions of a finite group on the derived category of an algebra or variety which does not necessarily come from an action on the algebra or variety itself. Inspired by Sosna \cite{Sosna}, we ask ourselves if such actions are always `abelian' in the sense that there is an invariant t-structure and hence an invariant heart. A positive answer will lead to a t-structure on the equivariant category.

This is a special situation of the following: Let $(\sX_i,\sY_i)$ be t-structures in a triangulated category $\sT$ indexed by a finite set $I$. When can this set of t-structures be `averaged' into a t-structure in $\sT$? Natural candidates for `averaged' t-structures are given by either taking the extension closure of the aisles, or the intersection of the aisles and seeing whether this is an aisle. Thus, we have the following natural questions:
\begin{itemize}
\item Is the extension closure of a finite set of aisles again an aisle?
\item Is the intersection of a finite set of aisles again an aisle?
\end{itemize}
Unfortunately, neither of the above questions have a positive answer in general. Thus, a further natural question is to describe exactly when this happens. As an aside, positive answers to the corresponding questions for co-aisles do not yield t-structures; see Example~\ref{ex:no-t-structure}.

In each case, it is easy to see that the required closure under (de)suspension and semi-orthogonality are always satisfied. The crux of the problem is to verify the generating condition, that is, to show the existence of truncation triangles.

For `big' categories, the answer is rather straightforward:

\begin{introtheorem} \label{thm:big}
Let $\sT$ be a triangulated category with set-indexed coproducts and $(\sX_i,\sY_i)_{i\in I}$ a finite set of compactly generated t-structures in $\sT$. Then:
\begin{itemize}
\item[(a)] The extension closure of the aisles (resp.\ co-aisles) is an aisle (resp.\ co-aisle);
\item[(b)] The intersection of the aisles (resp.\ co-aisles) is an aisle (resp.\ co-aisle).
\end{itemize}
\end{introtheorem}

The proof of Theorem~\ref{thm:big} makes use of the following na\"ive algorithm (see Section~\ref{sec:compact}):
Write $I=\{0,\ldots,d-1\}$ by choosing an auxiliary order on $I$. Given an object $t\in\sT$, we define two sequence of objects, $\naiveX{t}{n}$ and $\naiveY{t}{n}$ which fit into distinguished triangles 
 $\naiveX{t}{n} \to t \to \naiveY{t}{n} \to \Sigma\naiveX{t}{n}$ with $\naiveX{t}{n}\in\sX^I$, the extension closure of the $\sX_i$.
For $n=0$, this is just the truncation triangle of $t$ with respect to the t-structure $(\sX_0,\sY_0)$. For the recursion, define $\naiveY{t}{n}$ to be the left truncation of $\naiveY{t}{n-1}$ with respect to $(\sX_n,\sY_n)$ and then define $\naiveX{t}{n}$ using the above triangle.
We say that the algorithm \emph{converges} if the homotopy colimit of the sequence $\naiveY{t}{n}$ is in $\sY^I$, the intersection of the $\sY_i$.
Convergence is a straightforward application of machinery from \cite{AJS} or \cite{BR}.

For small triangulated categories, many of which are of interest in algebraic geometry and representation theory, matters are more subtle. Indeed, there are cases where the extension closure of aisles is not again an aisle; see Example~\ref{ex:no-t-structure}, and also examples where, even when the extension closure of aisles is an aisle, the na\"ive algorithm suggested by Theorem~\ref{thm:big} fails to compute the truncation triangles.
In these small triangulated categories, the absence of homotopy colimits means that instead of asking for convergent sequences, we have to ask for eventually constant ones (`termination');
in Section~\ref{sec:refined_algorithm} we refine the na\"ive algorithm of Section~\ref{sec:compact} by stripping off some irrelevant direct summands which can prevent the algorithm from terminating. In this section we prove the following main result:

\begin{introtheorem} \label{thm:refined}
Let $\sT$ be a $\kk$-linear, Krull-Schmidt, triangulated category and suppose that $(\sX_i,\sY_i)_{i\in I}$ is a finite set of t-structures in $\sT$. If the refined truncation algorithm of Section~\ref{sec:refined_algorithm} terminates then the extension closure $\sod{\sX_i \mid i\in I}$ of the aisles $\sX_i$ is an aisle.
\end{introtheorem}

The second part of the paper concerns the converse of Theorem~\ref{thm:refined}. In the case that $\sT$ is a piecewise hereditary triangulated category of tame representation type, we show that termination of the refined algorithm is both a necessary and sufficient condition for a finite set of aisles to be again an aisle. Furthermore, we are able to characterise the termination of the refined algorithm by an easily checkable combinatorial criterion. The main result of this part is (for the precise statement, see Theorem~\ref{thm:small:extended}):

\begin{introtheorem} \label{thm:small}
Let $\sT$ be a piecewise hereditary triangulated category of tame domestic type. Suppose  $(\sX_i,\sY_i)_{i\in I}$ is a finite set of compactly generated t-structures in $\sT$. Then, the following conditions are equivalent:
\begin{enumerate}
\item[(a)] The extension closure of the aisles is an aisle;
\item[(b)] The refined truncation algorithm of Section~\ref{sec:refined_algorithm} terminates;
\item[(c)] An easily checkable combinatorial criterion on the non-regular components holds.
\end{enumerate}
\end{introtheorem}

The combinatorial criterion needs to be checked only on the easier-to-understand non-regular component of the Auslander-Reiten quiver of $\sT$; this makes the authors believe that the results can be extended to all piecewise hereditary triangulated categories.

Theorem~\ref{thm:small} applies to (bounded derived categories of) path algebras of extended Dynkin quivers, and Deligne-Mumford curves of genus zero with non-positive orbifold Euler number (also known as `weighted projective lines of tame domestic type').

\smallskip
\noindent
\textbf{Acknowledgments.} We would like to thank Michael Gr\"ochenig, Martin Kalck and Dong Yang for useful comments.


\section{Preliminaries}\label{sec:prelim}

\noindent
Throughout this article $\sT$ will be a triangulated category with suspension functor (sometimes called translation or shift functor) $\Sigma\colon\sT\to\sT$. At various points, we shall introduce restrictions on $\sT$. We recall the definition of a t-structure from \cite[Definition 1.3.1]{BBD}:

\begin{definition} 
A \emph{t-structure} in a triangulated category $\sT$ consists of a pair of full subcategories $(\sX,\sY)$ satisfying the following axioms:
\begin{itemize}
\item[(i)]   $\Sigma\sX\subseteq \sX$ (and hence $\Sigma^{-1}\sY\subseteq \sY$);
\item[(ii)]  $\Hom_{\sT}(\sX,\sY)=0$;
\item[(iii)] For each object $t\in\sT$, there exists a distinguished triangle $\tri{t_{\sX}}{t}{t_{\sY}}$ with $t_{\sX}\in\sX$ and $t_{\sY}\in\sY$.
\end{itemize}
\end{definition}

The subcategory $\sX$ is called the \emph{aisle} of the t-structure, and $\sY$ is called the \emph{co-aisle}; see \cite[Section 1]{Keller-Vossieck}. The \emph{truncation triangle} in (iii) is uniquely determined and depends functorially on $t$; $t_{\sX}$ is called the \emph{right truncation} of $t$ and $t_{\sY}$ is called the \emph{left truncation}.

In a t-structure the full subcategories $\sX$ and $\sY$ are closed under extensions and taking direct summands. Furthermore, the aisle determines the co-aisle and vice versa in view of $\sX^\perp=\sY$ and ${}^\perp\sY=\sX$. We introduce these orthogonal subcategories and various others subcategories made from some $\sC\subset\sT$:

\medskip
\noindent
\begin{tabular}{@{} p{0.13\textwidth} @{} p{0.87\textwidth} @{}}
$\sC^\perp$,   & the \emph{right orthogonal} to $\sC$, the full subcategory of $t\in\sT$ with $\Hom(\sC,t)=0$, \\
${}^\perp\sC$, & the \emph{left orthogonal} to $\sC$, the full subcategory of $t\in\sT$ with $\Hom(t,\sC)=0$, \\
$\sod{\sC}$,   & the smallest full subcategory of $\sT$ containing $\sC$ that is closed under extensions 
                 and direct summands, \\
$\add(\sC)$,   & the smallest full, additive subcategory of $\sT$ containing $\sC$, \\
$\Susp(\sC)$,  & the \emph{(big) suspended subcategory generated by $\sC$}, the smallest full 
                 subcategory of $\sT$ containing $\sC$ which is closed under suspension, extensions,
                 set-indexed coproducts and taking direct summands, \\
$\ind(\sC)$,   & the set of \emph{indecomposable} objects of $\sC$ (up to isomorphism).
\end{tabular}

\begin{setup}
Let $\kk$ be a field and $\sT$ a $\kk$-linear triangulated category. 
Let $(\sX_i,\sY_i)_{i\in I}$ be a finite set of t-structures in $\sT$ and write $I=\{0,\ldots,d-1\}$.
Define two pairs of subcategories $(\sX^I,\sY^I)$ and $(\sX_I,\sY_I)$ by 
\[ \begin{array}{rcl @{\qquad} rcl}
 \sX^I &:=& \sod{\sX_i \mid i\in I}, & \sY^I &:=& \bigcap_{i\in I} \sY_i \\
 \sX_I &:=& \bigcap_{i\in I} \sX_i,   & \sY_I &:=& \sod{\sY_i \mid i\in I} .
\end{array} \]
\end{setup}

We have the following natural question: Are $(\sX_I,\sY_I)$ and $(\sX^I,\sY^I)$ t-structures in $\sT$? Note, in each case the required closure under (de)suspension, and the orthogonality condition, are always satisfied. The problem is to construct truncation triangles for each object of $\sT$.

The main results of this paper are stated and proved in terms of the averaged pair $(\sX^I,\sY^I)$. Dual statements can be formulated and proved for the other pair $(\sX_I, \sY_I)$. However, we shall refrain from stating them explicitly.


\section{The na\"ive truncation algorithm and Theorem~\ref{thm:big}} \label{sec:compact}

\noindent
In this section, we write down an algorithm which tries to compute the truncation triangle of an object with respect to $(\sX^I,\sY^I)$. We show that this always converges when applied to a finite set of compactly generated t-structures in a triangulated category with set-indexed coproducts.

\begin{naive-algorithm}
Given an object $t \in \sT$, we produce two sequences $\naiveX{t}{n}$ and $\naiveY{t}{n}$ of objects, with distinguished triangles $\naiveX{t}{n} \to t \to \naiveY{t}{n} \to \Sigma\naiveX{t}{n}$. By construction, all objects $\naiveX{t}{n}$ will lie in $\sX^I$.

\subsubsection*{Initial step}
Apply the t-structure $(\sX_0, \sY_0)$ to obtain a decomposition triangle
\[ t_{\sX_0} \too t \too t_{\sY_0} \too \Sigma t_{\sX_0} \]
for $t$. Then $t_{\sX_0}\in\sX_0\subset\sX^I$, so we set $\naiveX{t}{0}:=t_{\sX_0}$ and $\naiveY{t}{0}:=t_{\sY_0}$.

\subsubsection*{Iterative step}
Suppose we have a triangle
\[ \naiveX{t}{n-1} \too t \too \naiveY{t}{n-1} \too \Sigma\naiveX{t}{n-1} \]
with $\naiveX{t}{n-1}$ in $\sX^I$. We use the t-structure $(\sX_n,\sY_n)$, where $n$ is considered modulo $d$ and set $\naiveY{t}{n}:= (\naiveY{t}{n-1})_{\sY_n}$ to be the left truncation of $\naiveY{t}{n-1}$. Next, we define $\naiveX{t}{n}$ to be the cocone of the composition
 $t \to  \naiveY{t}{n-1} \to \naiveY{t}{n}$.
Using the octahedral axiom, this fits into a commutative diagram whose rows and columns are triangles:
\[ \xymatrix{
\naiveX{t}{n-1}  \ar@{=}[d] \ar[r] & \naiveX{t}{n} \ar[d] \ar[r] & (\naiveY{t}{n-1})_{\sX_n} \ar[d] \\
\naiveX{t}{n-1}             \ar[r] & t             \ar[d] \ar[r] & \naiveY{t}{n-1}          \ar[d] \\
                                   & \naiveY{t}{n}    \ar@{=}[r] & (\naiveY{t}{n-1})_{\sY_n}
} \]
From the top row of the diagram, we see that $\naiveX{t}{n}$ is an object in $\sX^I$.
\end{naive-algorithm}

In order to define convergence of the na\"ive algorithm, we need the following standard construction; see \cite{Neeman} for details.

\begin{definition} \label{def:homotopy-colimit}
Let $\sT$ be a triangulated category with set indexed coproducts and suppose we have a sequence of objects and morphisms, 
\[ t_0 \rightlabel{f_0} t_1 \rightlabel{f_1} t_2 \rightlabel{f_2} t_3 \rightlabel{f_3} \cdots . \]
The \emph{homotopy colimit} of this sequence is defined as the third object of the triangle
\[ \xymatrix@C=2.5em{
       \coprod_{i=0}^\infty t_i \ar[r]^{1-f_\bullet} & \coprod_{i=0}^\infty t_i \ar[r] & \hocolim{t_i} \ar[r] &
 \Sigma \coprod_{i=0}^\infty t_i.
} \]
\end{definition}

\begin{definition}
We say that the na\"ive algorithm \emph{terminates} for $t$ if there exists $n_0 \in \IN$ such that the sequence $(\naiveY{t}{n})$ is constant for $n \geq n_0$. 

We say that the na\"ive algorithm \emph{converges} for $t$ if the homotopy colimit of the sequence 
\[
\naiveY{t}{0} \to \naiveY{t}{1} \to \naiveY{t}{2} \to \cdots,
\]
$\hocolim{\naiveY{t}{n}}$ lies in $\sY^I$. Note that if the na\"ive algorithm terminates then it also converges. 
\end{definition}

Recall that an object $c\in\sT$ is called \emph{compact} if for any set-indexed family, $\{t_j\}_{j\in J}$, of objects of $\sT$ there is a canonical isomorphism
\[ \Hom_\sT(c, \coprod_{j\in J} t_j) \rightiso \coprod_{j\in J} \Hom_\sT(c, t_j) .\]

The next result is, independently, \cite[Theorem A.1]{AJS} and  \cite[Theorem III.2.3]{BR}.

\begin{thmdefn} \label{thm:compact-induced}
Let $\sT$ be a triangulated category with set-indexed coproducts and suppose that $\sC$ is a set of compact objects of $\sT$. Then the following pair, $(\sX,\sY)$, of full subcategories determines a t-structure on $\sT$:
\[
\sX = \Susp(\sC), \qquad
\sY = \{t\in\sT \mid \Hom_\sT(c, \Sigma^n t)=0 \text{ for all } c\in\sC \text{ and } n\leq 0\}.
\]
Such a t-structure is called \emph{compactly generated}.
\end{thmdefn}

We will use the following lemma in the proof of Theorem~\ref{thm:big}. Its proof follows easily from the argument in \cite[Theorem III.2.3]{BR}; we include an argument for the convenience of the reader. 
\begin{lemma} \label{lem:colimit-closure}
Let $(\sX,\sY)$ be a compactly generated t-structure in a triangulated category $\sT$ with set-indexed coproducts. Then the co-aisle $\sY$ is closed under homotopy colimits.
\end{lemma}

\begin{proof}
Let $\sC$ be a set of compact objects generating $(\sX,\sY)$. Then the right aisle $\sY$ is defined by
 $\sY  =  \{t\in\sT \mid \Hom_\sT(c, \Sigma^n t)=0 \text{ for all } c\in\sC \text{ and } n\leq 0\}$.
Consider a sequence of objects and morphisms in $\sY$:
\[ y_0 \rightlabel{f_0} y_1 \rightlabel{f_1} y_2 \rightlabel{f_2} y_3 \rightlabel{f_3} \cdots . \]
Now, since $\sC$ consists of compact objects, there is an isomorphism 
\[ \Hom_\sT(c,\Sigma^n\hocolim{y_i}) \cong \colim\Hom_\sT(c,\Sigma^n y_i)=0 \]
for all $n\leq 0$ and all $i\geq 0$, where the first isomorphism is by \cite[Lemma 2.8]{Neeman} and the second equality by the definition of $\sY$. It follows that $\hocolim{y_i}\in\sY$.
\end{proof}

\begin{proof}[Proof of Theorem~\ref{thm:big}]
Let $(\sX_i,\sY_i)$ be a family of t-structures in $\sT$ indexed by the finite set $I=\{0,\ldots,d-1\}$ for some $d\in\IN$. Let $t\in\sT$ and apply the na\"ive algorithm to obtain a tower of distinguished triangles.
\begin{equation} \label{eqn:tower}
\xymatrix@R=4ex{
\naiveX{t}{0} \ar[r]\ar[d] & t \ar[r]\ar@{=}[d] & \naiveY{t}{0} \ar[r]\ar[d] & \Sigma\naiveX{t}{0}\ar[d] \\
\naiveX{t}{1} \ar[r]\ar[d] & t \ar[r]\ar@{=}[d] & \naiveY{t}{1} \ar[r]\ar[d] & \Sigma\naiveX{t}{1}\ar[d] \\
\vdots                     & \vdots             & \vdots                     & \vdots 
}
\end{equation}
By construction $\naiveX{t}{n}\in\sX^I$ and $\naiveY{t}{n}\in\sY_n$, where $n$ is interpreted modulo $d$. 
Consider the subsequences 
\[
\naiveY{t}{i} \too \naiveY{t}{i+d} \too \naiveY{t}{i+2d} \too \naiveY{t}{i+3d} \too \cdots
\]
with objects in $\sY_i$ for each $i\in I$. It follows from \cite[Lemma 1.7.1]{Neeman-book} that $\hocolim{\naiveY{t}{i+md}} \cong \hocolim{\naiveY{t}{n}}$ and Lemma~\ref{lem:colimit-closure} implies that $\hocolim{\naiveY{t}{i+md}}\in\sY_i$ for each $i\in I$. Therefore, we see that $\hocolim{\naiveY{t}{n}}\in\sY^I$. The maps $t\to\naiveY{t}{n}$ give rise to a map $t\to\hocolim{\naiveY{t}{n}}$, which we extend to a distinguished triangle:
\[ \tri{x}{t}{\hocolim{\naiveY{t}{n}}} .\]

We now need to verify that $x\in\sX^I$. This is clear, because the tower \eqref{eqn:tower} produces a $3\times 3$ diagram of distinguished triangles by \cite[Proposition 1.1.11]{BBD}:
\[
\xymatrix@R=4ex{
 \coprod_{n=0}^\infty \naiveX{t}{n} \ar[r] \ar[d]               & \coprod_{n=0}^\infty t                     \ar[r] \ar[d]^-{1-\mathrm{id}_\bullet}  & 
 \coprod_{n=0}^\infty \naiveY{t}{n} \ar[r] \ar[d]^-{1-f_\bullet} & \Sigma \coprod_{n=0}^\infty \naiveX{t}{n}         \ar[d] \\
 \coprod_{n=0}^\infty \naiveX{t}{n} \ar[r] \ar[d]               & \coprod_{n=0}^\infty t                     \ar[r] \ar[d] &
 \coprod_{n=0}^\infty \naiveY{t}{n} \ar[r] \ar[d]               & \Sigma \coprod_{n=0}^\infty \naiveX{t}{n}         \ar[d] \\
 x                                \ar[r] \ar[d]               & t                                         \ar[r] \ar[d] &
 \hocolim{\naiveY{t}{n}}          \ar[r] \ar[d]               & \Sigma x                                         \ar[d] \\
 \Sigma\coprod_{n=0}^\infty \naiveX{t}{n}  \ar[r]               & \Sigma\coprod_{n=0}^\infty t               \ar[r]         &
 \Sigma\coprod_{n=0}^\infty \naiveY{t}{n}  \ar[r]               & \Sigma^2\coprod_{n=0}^\infty \naiveX{t}{n}
}
\]
Since $\sX^I$ is closed under coproducts, extensions and suspension, it follows that $x$ in the first vertical distinguished triangle also lies in $\sX^I$. This completes the proof of Theorem~\ref{thm:big}.
\end{proof}


\section{The refined truncation algorithm and Theorem~\ref{thm:refined}} \label{sec:refined_algorithm}

\noindent
For applications in small triangulated categories, i.e.\ those which do not possess all set-indexed coproducts, it does not make sense to talk about the convergence of the na\"ive algorithm, in general.
One can only sensibly ask for the stronger condition that it terminates.
However, we can have a finite set of t-structures for which the extension closure of the aisles is an aisle, but for which the na\"ive algorithm does not terminate. This is made clear in the following simple example.

\begin{example} \label{ex:naive-algorithm-fails}
Let $\sT = \Db(\kk A_2)$ be the bounded derived category of right modules over the path algebra $\kk A_2$ of the $A_2$ quiver. It is well known that the indecomposable objects of $\sT$ are given by the simple-projective $S(1)=P(1)$, the projective-injective $P(2)=I(1)$, the simple-injective $S(2)=I(2)$, and all (de)suspensions of those. 
Consider the t-structure given by
$\sX_1 = \add(\Sigma^i P(2) \mid i\in\IZ)$ and $\sY_1 = \add(\Sigma^i P(1) \mid i\in\IZ)$. Likewise, there are t-structures $(\sX_2,\sY_2)$ using $I(2)$ and $I(1)$ instead of the projectives, and $(\sX_3,\sY_3)$ using $S(1)$ and $S(2)$.
They are drawn in Figure~\ref{fig:A2_t-structures}. The na\"ive algorithm does not terminate, but the extension closure of the aisles is $\Db(\kk A_2)$, and hence an aisle.
\begin{figure}
\centering
\includegraphics[width=0.8\textwidth]{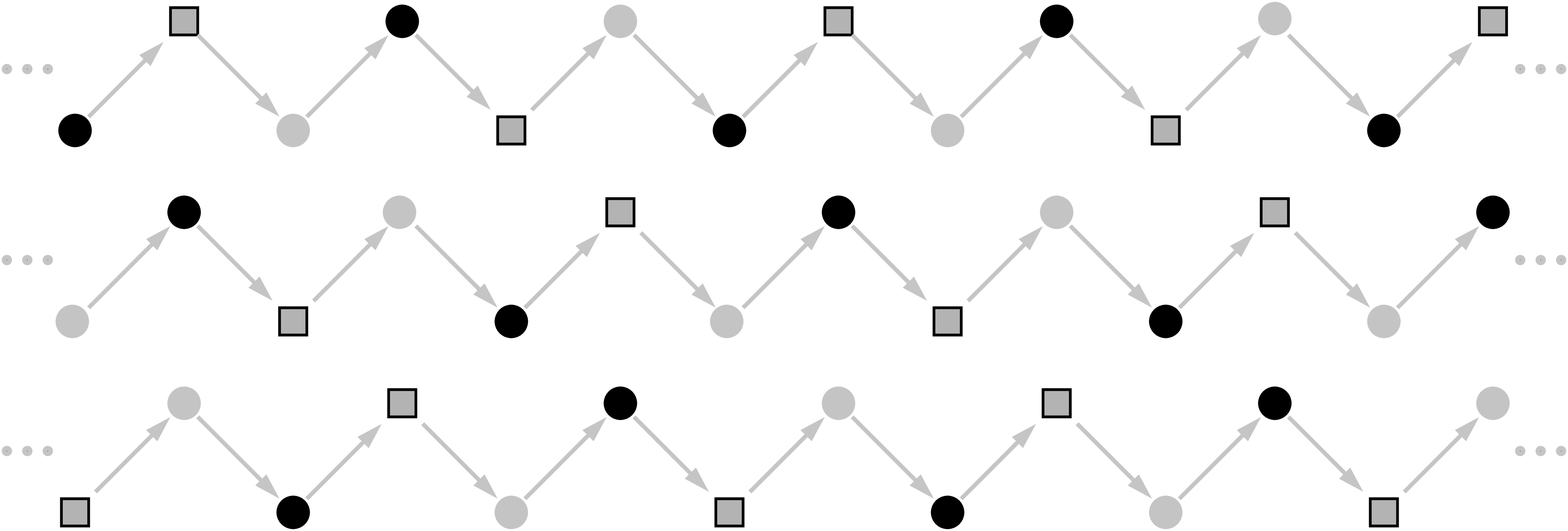}
\caption[Three t-structures on $\sT=\Db(\IX(2,2))$.]{ \label{fig:A2_t-structures} \small
   Three t-structures on $\sT=\Db(\kk A_2)$ defined by $\sX$ (\dotspace\boxeddot) and $\sY$ (\dotspace\blackdot); 
   the remaining indecomposables (\dotspace\greydot) belong to neither $\sX$ nor $\sY$.}
\end{figure}
\end{example}

\begin{remark} 
Observe that in this example the morphisms $\naiveY{t}{n} \to \naiveY{t}{n+2}$ are zero for all $n\in\IN$. In particular,
$\Hom(t,\naiveY{t}{n})=0$ for $n\geq1$. In fact, if we were to work in the corresponding `big' category, these zero morphisms would imply that the homotopy colimit of the sequence $(\naiveY{t}{n})$ is zero (consider the subsequence $(\naiveY{t}{2n})$ with zero morphisms), which is actually an object in the `small' category. The additional summands with a zero morphism from $t$ prevent the algorithm from terminating, but do not contribute anything to the `limit'. This observation motivates the following refinement of the algorithm.
\end{remark}

\begin{setup}
Let $\kk$ be a field and $\sT$ be a $\kk$-linear, Krull-Schmidt triangulated category, i.e.\ all idempotents of $\sT$ split, each object of $\sT$ is a direct sum of finitely many indecomposable objects, and each indecomposable object has local endomorphism ring.
\end{setup}

For example, bounded derived categories of $\kk$-linear, Hom-finite abelian categories are Krull-Schmidt.

\begin{refined-algorithm} \label{alg:new}
Let $t$ be an object of $\sT$. Again, we construct two sequences of objects $\refinedX{t}{n}$ and $\refinedY{t}{n}$, fitting into distinguished triangles
 $\refinedX{t}{n} \to t \to \refinedY{t}{n} \to \Sigma\refinedX{t}{n}$
where $\refinedX{t}{n}$ is in $\sX^I$ but with the additional property that the composition of $t \to \refinedY{t}{n}$ with the projection onto each summand of $\refinedY{t}{n}$ is non-zero.

\subsubsection*{Initial step.}
We truncate $t$ with respect to the t-structure $(\sX_0,\sY_0)$ to get the distinguished triangle
\[ t_{\sX_0} \too t \rightlabel{\alpha_0} t_{\sY_0} \too \Sigma t_{\sX_0} .\]
Note that $t_{\sX_0}\in\sX_{0}\subset\sX^I$ and the morphism $\alpha_0$ is non-zero to all summands of $t_{\sY_0}$.
Thus $\refinedX{t}{0}:=t_{\sX_0}$ and $\refinedY{t}{0}:=t_{\sY_0}$ have the desired properties.

\subsubsection*{Iterative step.} 
We have a triangle
\[ \refinedX{t}{n-1} \too t \xrightarrow{\alpha_{n-1}} \refinedY{t}{n-1} \too \Sigma \refinedX{t}{n-1} \]
with $\refinedX{t}{n-1}\in\sX^I$ and such that
$\alpha_{n-1}$ is non-zero on any summand of $\refinedY{t}{n-1}$.
Truncate $\refinedY{t}{n-1}$ with respect to the t-structure $(\sX_n,\sY_n)$, where the index $n$ is interpreted modulo $d$, to get the triangle
\begin{equation}  \label{eqn:y_n}
(\refinedY{t}{n-1})_{\sX_n} \to \refinedY{t}{n-1} \to (\refinedY{t}{n-1})_{\sY_n} \to \Sigma(\refinedY{t}{n-1})_{\sX_n} .
\end{equation}
Now consider the composition
\[ \xymatrix{
   t \ar[r]^-{\alpha_{n-1}} \ar[dr]_-{\tilde{\alpha}_n} & \refinedY{t}{n-1} \ar[d] \\
                                                       & (\refinedY{t}{n-1})_{\sY_n} .
} \]
Using the Krull-Schmidt hypothesis on $\sT$, we decompose $(\refinedY{t}{n-1})_{\sY_n}$
\[ \xymatrix{
 {}\save[]+<0em,0em>*{(\refinedY{t}{n-1})_{\sY_n} =} \restore
 & \refinedY{t}{n} \oplus \waste{t}{n}\phantom{X} \ar[dl]_{p_n} \ar[dr]^{\pi_n} \\
 \refinedY{t}{n}             &                                     & \waste{t}{n}
} \]
where $p_n$ and $\pi_n$ are projections, in the unique way such that
\begin{align*}
 \alpha_n := p_n \tilde\alpha_n &\colon t \to (\refinedY{t}{n-1})_{\sY_n} \to \refinedY{t}{n}
                                     \text{ is non-zero onto all summands of $\refinedY{t}{n}$}, \\
 \pi_n \tilde\alpha_n &\colon t \to (\refinedY{t}{n-1})_{\sY_n} \to \waste{t}{n} \text{ is zero.}
\end{align*}
Now define $\refinedX{t}{n}$ to be the cocone of $\alpha_n$ to get the distinguished triangle
\begin{equation}  \label{eqn:x_n}
\refinedX{t}{n} \too t \rightlabel{\alpha_n} \refinedY{t}{n} \too \Sigma \refinedX{t}{n}
\end{equation}
The morphism $\alpha_n$ satisfies the non-vanishing property by construction, so we only need to check that $\refinedX{t}{n}\in\sX^I$; this is contained in the main result of this section.
\end{refined-algorithm}

This was called Theorem~\ref{thm:refined} in the introduction:

\begin{theorem} \label{thm:x_n}
Let $\sT$ be a $\kk$-linear, Krull-Schmidt triangulated category. If the refined algorithm terminates then $\sX^I$ is an aisle.
\end{theorem}

\begin{proof} 
We shall actually prove that the object $\refinedX{t}{n}$ in the triangle \eqref{eqn:x_n} above lies in $\sX^I$ for all $n \in \IN$. The conclusion that $\sX^I$ is an aisle follows directly from this and the fact that upon termination of the refined algorithm, the object $\refinedY{t}{n}$ lies in $\sY^I$. Thus, the triangle \eqref{eqn:x_n} is the required truncation triangle.
The proof is elaborate and we proceed in a number of steps.

\noindent
\begin{tabular}{ @{} p{0.1\textwidth} @{} p{0.9\textwidth}}
Step 1: & $\refinedX{t}{n}\in\sX^I$ follows from $\waste{t}{n}\in\sX^I$. \\
Step 2: & If the natural map $w\to\naiveY{w}{n}$ is zero for $w\in\sT$ and some $n$, then $w\in\sX^I$. \\
Step 3: & The map $\waste{t}{n}\to\naiveY{\waste{t}{n}}{n}$ is zero.
\end{tabular}
\noindent
The vanishing proved in Step 3 allows us to invoke Step 2, which shows that $\waste{t}{n}\in\sX^I$. The result then follows directly from Step 1.

\medskip\noindent
\emph{Step 1.}
Assuming that $\waste{t}{n}\in\sX^I$, the following two octahedra give $\refinedX{t}{n}\in\sX^I$:

\[ \xymatrix@C=1em{
   \refinedX{t}{n-1} \ar[r] \ar@{=}[d] & u                        \ar[d] \ar[r] & (\refinedY{t}{n-1})_{\sX_n} \ar[d] \\
   \refinedX{t}{n-1} \ar[r]            & t                        \ar[d] \ar[r] & \refinedY{t}{n-1}         \ar[d] \\
                                       & (\refinedY{t}{n-1})_{\sY_n}   \ar@{=}[r] & (\refinedY{t}{n-1})_{\sY_n}        
}
\qquad
\xymatrix@C=2em{
   u \ar[r] \ar@{=}[d]                 & \refinedX{t}{n}          \ar[d] \ar[r] & \waste{t}{n}              \ar[d] \\
   u \ar[r]                            & t                        \ar[d] \ar[r] & (\refinedY{t}{n-1})_{\sY_n} \ar[d] \\
                                       & \refinedY{t}{n}             \ar@{=}[r] & \refinedY{t}{n} 
} \]

The auxiliary object $u$ defined by the left-hand diagram is in $\sX^I$, as the top row shows. Invoking the assumption $\waste{t}{n}\in\sX^I$, the top row of the right-hand diagram then gives $\refinedX{t}{n}\in\sX^I$.

\medskip\noindent
\emph{Step 2.}
We look at the na\"ive algorithm applied to an object $w$ (later, we will use $w=\waste{t}{n}$). The result is the following tower of triangles (i.e.\ a commutative diagram in which all horizontal triangles are distinguished)
\[ \xymatrix@R=3ex@!C{
\naiveX{w}{0} \ar[r] \ar[d] & w \ar[r]^-{f_0} \ar@{=}[d] & \naiveY{w}{0} \ar[r] \ar[d] & \Sigma\naiveX{w}{0} \ar[d] \\
\naiveX{w}{1} \ar[r] \ar[d] & w \ar[r]^-{f_1} \ar@{=}[d] & \naiveY{w}{1} \ar[r] \ar[d] & \Sigma\naiveX{w}{1} \ar[d] \\
\vdots               \ar[d] & \vdots         \ar@{=}[d] & \vdots               \ar[d] & \vdots              \ar[d] \\
\naiveX{w}{n} \ar[r]        & w \ar[r]^-{f_n}           & \naiveY{w}{n} \ar[r]        & \Sigma\naiveX{w}{n} \\
} \]
Recall that by construction, each $\naiveY{w}{n}$ is in $\sY_n$. The assumption $f_n=0$ means that the bottom triangle splits as $\naiveX{w}{n} \cong w \oplus \Sigma^{-1}\naiveY{w}{n}$. With $\naiveX{w}{n} \in\sX^I$ and $\sX^I$ closed under summands, we get $w\in\sX^I$.

\medskip\noindent
\emph{Step 3.}
For arbitrary $t\in\sT$, we prove the following statement: The natural map $\waste{t}{n}\to\naiveY{\waste{t}{n}}{n}$ is zero. Put $w:=\waste{t}{n}$ for brevity.

The claim is trivial for $n=0$ as $\refinedX{t}{0}=\naiveX{t}{0}\in\sX_0\subset\sX^I$ and $\waste{t}{0}=0$.
For any $n>0$, we consider the starting triangle together with a chain of maps composing to $\refinedY{t}{0}\to\naiveY{w}{0}$:

\[ \xymatrix@R=4ex@!C{
    \refinedX{t}{0}     \ar[r]                      & t      \ar[r]^-{\alpha_0} \ar@{-->}@/_2.5pc/[6,1]_0 &
    \refinedY{t}{0}     \ar[r] \ar[d]^{\tilde\alpha_1} & \Sigma \refinedX{t}{0} \ar@{-->}@/^2.5pc/[6,-1]^\exists \\
& & \refinedY{t}{1}            \ar[d]^{\tilde\alpha_2} \\
& & \vdots                     \ar[d]^{\tilde\alpha_{n-1}} \\
& & \refinedY{t}{n-1}          \ar[d]^{\alpha_n} \\
& & (\refinedY{t}{n-1})_{\sY_n} \ar[d]^{\pi_n} \\
& & w                          \ar[d]^(0.3){f_0} \\
& & \naiveY{w}{0}
} \]

The broken arrow labelled 0 vanishes by construction of $w=\waste{t}{n}$. Therefore, the broken arrow labelled $\exists$ exists. However, this map is zero because of $\Sigma\refinedX{t}{0}\in\sX_0$ and $\naiveY{w}{0}\in\sY_0$. This means that the composition of all the vertical arrows is zero. We will set up a similar diagram by rotating one piece of the top most vertical arrow into a distinguished triangle (recall $\tilde\alpha_1=p_1\alpha_1$):
\[ \xymatrix@R=4ex{
    (\refinedY{t}{0})_{\sX_1}                \ar[r]  & \refinedY{t}{0} \ar[r]^{\alpha_1} \ar@{-->}@/_2.2pc/[5,1]_0 &
    (\refinedY{t}{0})_{\sY_1}    \ar[d]^{p_1} \ar[r]  & \Sigma (\refinedY{t}{0})_{\sX_1} \ar@{-->}@/^2.5pc/[6,-1]^\exists \\
& & \refinedY{t}{1}             \ar[d]^{\tilde\alpha_{n-1}\cdots\tilde\alpha_2} \\
& & \refinedY{t}{n-1}           \ar[d]^{\alpha_n} \\
& & (\refinedY{t}{n-1})_{\sY_n}  \ar[d]^{\pi_n} \\
& & w                         \ar[d]^(0.4){f_0} \\
& & \naiveY{w}{0}               \ar[d] \\
& & \naiveY{w}{1}                       
} \]
The left-hand broken arrow is 0 from the last diagram. Reasoning as before we see that the right-hand broken arrow exists and, as a map from $\sX_1$ to $\sY_1$, it is 0. So again, the composition of all vertical maps vanishes. Because this composition starts with a projection $p_1$, even the reduced composition $f_1\pi_n\alpha_n\tilde\alpha_{n-1}\cdots\tilde\alpha_2=0$. (Recall that the composition of the two bottom arrows is $f_1$.)

We can now iterate these two diagrams, by rotating the map $\alpha_2$ of the vertical composition into the horizontal triangle and appending the vertical chain by the map $\naiveY{w}{1}\to\naiveY{w}{2}$. In the next run, we find
 $f_2\pi_n\alpha_n\tilde\alpha_{n-2}\cdots\tilde\alpha_3 p_2=0$.
Again, we can strip off $p_2$ from this relation. Eventually, we arrive at $f_n\pi_n=0$, and it follows that the map $w\to\naiveY{w}{n}$ is 0 as required, since $\pi_n$ is a projection. 
\end{proof}

We end this section with a lemma relating terms of the refined algorithm when one starts at different points, and a lemma showing that the refined truncation algorithm commutes with direct sums.

\begin{lemma} \label{lem:summands}
Let $t$ be an object of $\sT$ and write $t_n = \refinedY{t}{n}$ for $n\geq -1$, where $t_{-1}=t$. Suppose that $k \equiv 1 \bmod d$.
Then the object $\refinedY{t}{k+n}$ is a direct summand of  
$\refinedY{t_k}{n}$ for all $n \geq 0$.
\end{lemma}

\begin{proof}
For  
$n=0$, the statement is clear. Suppose $n>0$
and $\refinedY{t_k}{n} =  
\refinedY{t}{n+k} \oplus b_n$. Taking the left truncation of  
$\refinedY{t_k}{n}$ gives:
\[
(\refinedY{t_k}{n})_{\sY_{n+1}} = \refinedY{t}{k+n+1} \oplus  
\waste{t}{k+n+1} \oplus (b_n)_{\sY_{n+1}} = \refinedY{t_k}{n+1} \oplus  
\waste{t_k}{n+1}.
\]
Now since the projection of the map $t\to t_k \to  
(\refinedY{t_k}{n})_{\sY_{n+1}}$ is, by definition, non-zero to the  
summands of $\refinedY{t}{k+n+1}$, it follows that the map $t_k \to  
(\refinedY{t_k}{n})_{\sY_{n+1}}$ is non-zero to the summands of  
$\refinedY{t}{k+n+1}$. Hence, the summands of $\refinedY{t}{k+n+1}$  
are in fact summands of $\refinedY{t_k}{n+1}$, as required.
\end{proof}

\begin{lemma}
The refined truncation algorithm commutes with finite direct sums, i.e.\
 $\refinedY{t \oplus t' }{n} = \refinedY{t}{n} \oplus \refinedY{t'}{n}$ and
 $\refinedX{t \oplus t' }{n} = \refinedX{t}{n} \oplus \refinedX{t'}{n}$ for all $n \in \IN$.
\end{lemma}

\begin{proof}
It is clear that direct sums commute with truncations since we may take the sum of the truncation triangles for $t$ and $ t'$ to obtain a truncation triangle for $t \oplus t'$. Now suppose for an inductive argument that $\refinedY{t \oplus t' }{n} = \refinedY{t}{n} \oplus \refinedY{t'}{n}$ and the morphism $t \oplus t' \to \refinedY{t}{n} \oplus \refinedY{t'}{n}$ is diagonal. Truncating with respect to $\sY_{n+1}$, we obtain a diagonal composition morphism $t \oplus t' \to (\refinedY{t}{n})_{\sY_{n+1}} \oplus (\refinedY{t'}{n})_{\sY_{n+1}}$. The morphism from $t \oplus t'$ to a summand of $ (\refinedY{t}{n})_{\sY_{n+1}}$ (respectively $(\refinedY{t'}{n})_{\sY_{n+1}}$) is zero if and only if the restriction from $t$ (respectively $t'$) is zero. 
\end{proof}

\section{Piecewise hereditary triangulated categories}\label{sec:background}

In this section, we collect some facts about the structure of hereditary triangulated categories. The material is standard and can be found in a number of sources, for example: \cite{ASS}, \cite{ARS}, \cite{Happel}, \cite{Lenzing}, \cite{Ringel} \cite{SS1} and \cite{SS2}. Readers familiar with this material may wish to skip this section.

For the rest of the paper $\kk$ will be an algebraically closed field. Recall that an abelian category is \emph{hereditary} if it is of global dimension 0 or 1, i.e.\ the $\Ext^n$ bifunctors are zero for $n\geq2$. The following definition can be found in \cite[p. 120]{Lenzing}, for example.

\begin{definition}
A $\kk$-linear triangulated category $\sT$ is called \emph{piecewise hereditary} if it is triangle equivalent to $\Db(\sH)$, for some $\kk$-linear, Hom- and Ext-finite, hereditary abelian category $\sH$ with a tilting object.
\end{definition}

\subsection{Finite and tame domestic types}

By a theorem of Happel \cite{Happel2}, a piecewise hereditary triangulated category is triangle equivalent to $\Db(\mod{\Lambda})$ for some finite dimensional hereditary $\kk$-algebra $\Lambda$, or $\Db(\coh{\IX})$ for some `weighted projective line' $\IX$. We write $\Db(\Lambda)=\Db(\mod{\Lambda})$ and $\Db(\IX)=\Db(\coh{\IX})$.
A coarse classification of these categories, up to triangle equivalence, is as follows --- see \cite[p.\ 126]{Lenzing} for a schematic:
\begin{itemize*}
\item[\emph{Finite type:}] $\sT\cong\Db(\kk Q_{ADE})$ for a quiver of simply-laced Dynkin type.
\item[\emph{Tame domestic type:}] $\sT\cong\Db(\kk {\tilde Q}_{ADE})\cong\Db(\IX_{\chi>0})$ for a quiver of extended simply-laced Dynkin type (a representative of the triangle equivalence class of tame hereditary algebras) or, equivalently, a weighted projective line of positive orbifold Euler characteristic.
\item[\emph{Tame tubular type:}] $\sT\cong\Db(\IX_{\chi=0})$ for an $\IX$ of zero orbifold Euler characteristic.
\item[\emph{Wild type:}] $\sT\cong\Db(\Lambda)$ for a wild hereditary algebra $\Lambda$, or $\sT\cong\Db(\IX_{\chi<0})$ for an $\IX$ of negative orbifold Euler characteristic.
\end{itemize*}

In this paper we restrict our attention to piecewise hereditary triangulated categories of \emph{finite} or \emph{tame domestic} type. In particular, such a category is triangle equivalent to $\Db(\Lambda)$ for a hereditary $\kk$-algebra $\Lambda$ of finite or tame representation type.

\subsection{Structure of the module category} \label{sec:module-category}

For a general hereditary algebra $\Lambda$, the category $\mod{\Lambda}$ is a Krull-Schmidt abelian category. Its \emph{Auslander-Reiten (AR) quiver} has as vertices isomorphism classes of indecomposable $\Lambda$-modules and as arrows irreducible $\Lambda$-module homomorphisms between indecomposable $\Lambda$-modules, i.e.\ those homomorphisms which do not factor through other homomorphisms. It can be represented graphically as follows, where non-zero morphisms only exist from left to right:

\begin{center}
\begin{tikzpicture}

\draw (0,0) -- (2.4,0);
\draw (0,0.8) -- (2.4,0.8);
\draw (0,0) -- (0, 0.8);
\draw [decorate,decoration=zigzag] (2.4,0) -- (2.4, 0.8);

\draw [thick,decorate,decoration={brace,mirror,raise=15pt}] (0,0) -- (2.4,0)
		node[pos=0.5,anchor=north,yshift=-0.55cm] {postprojectives};

\draw (3.5,0.4) circle (7mm);

\draw [thick,decorate,decoration={brace,mirror,raise=15pt}] (2.8,0) -- (4.2,0)
		node[pos=0.5,anchor=north,yshift=-0.55cm] {regular};

\draw (4.6,0) -- (7,0);
\draw (4.6,0.8) -- (7,0.8);
\draw [decorate,decoration=zigzag] (4.6,0) -- (4.6, 0.8);
\draw (7,0) -- (7, 0.8);

\draw [thick,decorate,decoration={brace,mirror,raise=15pt}] (4.6,0) -- (7,0)
		node[pos=0.5,anchor=north,yshift=-0.55cm] {preinjectives};

\end{tikzpicture}

\end{center}

The rectangle on the left hand side is called the \emph{postprojective} component $\cP$ containing the projective indecomposable modules. This component continues, in general, infinitely to the right. Dually, the rectangle on the right hand side is called the \emph{preinjective} component $\cI$, containing the injective indecomposable modules. The preinjective component continues, in general, infinitely, to the left. The central circle represents the so-called \emph{regular} components $\cR$, and as such, the postprojective component and the preinjective component are often referred to as the \emph{non-regular} components. Note that in the literature the postprojective component is often called the preprojective component.

The category $\mod{\Lambda}$ comes equipped with the Auslander-Reiten translate $\tau$, which is usually represented on the AR quiver as an arrow to the object immediately to the left in the AR quiver.

For $\Lambda$ of finite type, there are only finitely many indecomposable $\Lambda$-modules. In this case, there is no regular component and a unique non-regular component which is both postprojective and preinjective.

For $\Lambda$ of tame type, there are infinitely many indecomposable $\Lambda$-modules with all three components. The regular component takes the form of a $\IP^{1}$-indexed family of standard stable tubes of finite rank (see Subsection~\ref{sec:structure-tubes} for tubes).

\subsection{Structure of the bounded derived category}\label{sec:bdd-derived}

A standard lemma implies that each object of $\Db(\Lambda)$ decomposes as a direct sum of its cohomology; see \cite[Lemma I.5.2 and Corollary I.5.3]{Happel}. This means that the AR quiver of $\Db(\Lambda)$ takes the form:

\begin{center}
\begin{tikzpicture}

\draw (0,0) -- (2.4,0);
\draw (0,0.8) -- (2.4,0.8);
\draw [decorate,decoration=zigzag] (0,0) -- (0, 0.8);
\draw [decorate,decoration=zigzag] (2.4,0) -- (2.4, 0.8);
\draw (1.2,0) -- (1.2, 0.8);

\draw (3.2,0.4) circle (5mm);

\draw [thick,decorate,decoration={brace,mirror,raise=10pt}] (1.3,0) -- (5.1,0) node [pos=0.5,anchor=north,yshift=-0.55cm] {$\Sigma^{-1} \mod{\Lambda}$};

\draw (4.0,0) -- (6.4,0);
\draw (4,0.8) -- (6.4,0.8);
\draw [decorate,decoration=zigzag] (4.0,0) -- (4.0, 0.8);
\draw [decorate,decoration=zigzag] (6.4,0) -- (6.4, 0.8);
\draw (5.2,0) -- (5.2, 0.8);

\draw (7.2,0.4) circle (5mm);

\draw [thick,decorate,decoration={brace,mirror,raise=10pt}] (5.3,0) -- (9.1,0) node [pos=0.5,anchor=north,yshift=-0.55cm] {$\mod{\Lambda}$};

\draw (8.0,0) -- (10.4,0);
\draw (8.0,0.8) -- (10.4,0.8);
\draw [decorate,decoration=zigzag] (8.0,0) -- (8.0, 0.8);
\draw [decorate,decoration=zigzag] (10.4,0) -- (10.4, 0.8);
\draw (9.2,0) -- (9.2, 0.8);

\draw (11.2,0.4) circle (5mm);

\draw [thick,decorate,decoration={brace,mirror,raise=10pt}] (9.3,0) -- (13.1,0) node [pos=0.5,anchor=north,yshift=-0.55cm] {$\Sigma \mod{\Lambda}$};

\draw (12.0,0) -- (14.4,0);
\draw (12.0,0.8) -- (14.4,0.8);
\draw [decorate,decoration=zigzag] (12.0,0) -- (12.0, 0.8);
\draw [decorate,decoration=zigzag] (14.4,0) -- (14.4, 0.8);
\draw (13.2,0) -- (13.2, 0.8);

\end{tikzpicture}

\end{center}

\noindent Again morphisms go from left to right, and since for two modules $M,N\in\mod{\Lambda}$, one has $\Hom_{\Db(\Lambda)}(M,\Sigma^{n} N) \cong \Ext^{n}_{\Lambda}(M,N)$, this means that non-zero morphisms exist only from one degree to the next and not any higher. There is also an Auslander-Reiten translate on $\Db(\Lambda)$, still denoted by $\tau$.

The special structure of $\Db(\Lambda)$ means that cones and cocones are easily computed as follows: If $f \colon M \to N$ is a morphism in $\mod{\Lambda}$, then the cone of this map is $\Cone(f)=\coker(f) \oplus \Sigma \kernel(f)$.

There is a partial order on the indecomposable objects of the non-regular components of $\Db(\Lambda)$, namely, $x \leq y$ if and only if there is a chain of arrows in the AR quiver of $\Db(\Lambda)$ between the vertex representing the isomorphism class of $x$ to the vertex representing that of $y$.

\subsection{The structure of the regular component --- tubes} \label{sec:structure-tubes}

We list some pertinent facts about standard stable tubes; see \cite[\S X.2]{SS1} for details: Each such tube $\cT$ is uniquely determined by a natural number, its \emph{rank}. Note that the adjective `standard' here means that all the morphisms in $\cT$ can be read off from its AR quiver, and `stable' means that none of the objects of $\cT$ are projective or injective.
Such tubes are hereditary, $\kk$-linear, abelian categories with finite-dimensional $\Hom$ and $\Ext^1$ spaces. Furthermore, they are uniserial, i.e.\ every indecomposable object has a unique composition series. Therefore, every indecomposable object $t$ of the tube has a (composition) \emph{length} $\len{t}$. The objects of length 1, i.e.\ those with trivial composition series, are the so-called \emph{quasi-simple} objects $s_1,s_2,\ldots,s_\tuberk$, where $\tuberk$ is the rank of the tube. The terminology `quasi-simple' stems from the fact that the $s_i$ are simple as objects of the tube but they will generally not be simple in a larger abelian category containing the tube.

The quasi-simple objects form what is called the \emph{mouth} of the tube. In particular, we can describe indecomposables by listing the quasi-simple constituents of the composition series from top to socle. Also, being uniserial makes the category extension closed. Each indecomposable object $t$ determines the following three full subcategories of $\sT$:

\begin{itemize*}
\item[$\sR(t)$,] the \emph{ray} at $t$, whose objects all have the same quasi-simple socle, 
\item[$\sC(t)$,] the \emph{coray} at $t$, whose objects all have the same quasi-simple top, 
\item[$\sW(t)$,] the \emph{wing} below $t$, consisting of all indecomposables of length less than or equal to $\len{t}$ and trapped between the ray $\sR(t)$ and the coray $\sC(t)$; we assume that $\ell(t)$ is less than the rank of the tube;
\end{itemize*}
see \cite[Section 3.1]{Baur-Buan-Marsh} or \cite[Section X.1]{SS1}, for instance. Figure~\ref{fig:tube} shows a tube.

\begin{figure}
%
%
\setlength{\fboxsep}{-0.5ex}                                
\newcommand{\ubox}[1]{ \boxed{ \! #1 \! } }
\newcommand{\usize}{\scriptstyle}                           
\newlength{\unl}
\setlength{\unl}{-1.1ex}                                    
\newcommand{\uarDR}{{ \POS[];[dr]**\dir{-} ?> *\dir{>} }}   
\newcommand{\uarUR}{{ \POS[];[ur]**\dir{-} ?> *\dir{>} }}
\newcommand{\uarLL}{{ \POS[];[rr]**\dir{--} ?< *\dir{<} }}  
\newcommand{\uW}[4]{{  \ubox{\begin{array}{c} \usize #1 \\[\unl] 
                                              \usize #2 \\[\unl] 
                                              \usize #3 \\[\unl]
                                              \usize #4          \end{array}} \uarLL \uarUR \uarDR }} 
\newcommand{\uZ}[3]{{  \ubox{\begin{array}{c} \usize #1 \\[\unl]
                                              \usize #2 \\[\unl] 
                                              \usize #3          \end{array}} \uarLL \uarUR \uarDR }} 
\newcommand{\uY}[2]{{  \ubox{\begin{array}{c} \usize #1 \\[\unl] 
                                              \usize #2          \end{array}} \uarLL \uarUR \uarDR }} 
\newcommand{\uX}[1]{{  \ubox{\begin{array}{c} \usize #1                                       \end{array}} \uarLL \uarUR        }}
\newcommand{\uWW}[4]{{ \ubox{\begin{array}{c} \usize #1 \\[\unl] 
                                              \usize #2 \\[\unl] 
                                              \usize #3 \\[\unl]
                                              \usize #4          \end{array}} \uarUR \uarDR }} 
\newcommand{\uZZ}[3]{{ \ubox{\begin{array}{c} \usize #1 \\[\unl] \usize #2 \\[\unl] \usize #3 \end{array}}                      }}
\newcommand{\uYY}[2]{{ \ubox{\begin{array}{c} \usize #1 \\[\unl] \usize #2                    \end{array}}        \uarUR \uarDR }}
\newcommand{\uXX}[1]{{ \ubox{\begin{array}{c} \usize #1                                       \end{array}}                      }}
\newcommand{\uXx}[1]{{ \ubox{\begin{array}{c} \usize #1                                       \end{array}} \uarLL \uarUR        }}
\newcommand{\uvdots}{ \vdots \uarDR }
\begin{minipage}{0.65\textwidth}
\hspace*{-1em} 
$\xymatrix@M=0.1em@H=0.0ex@W=0.4ex@!0{
\uvdots    &         & \uvdots   &         & \uvdots    &         & \uvdots   &         & \uvdots    &          & \vdots     \\
           & \uW C B A E &       & \uW D C B A &        & \uW E D C B &       & \uW A E D C &        & \uWW B A E D &        \\
\uZ  B A E &         & \uZ C B A &         & \uZ D C B  &         & \uZ E D C &         & \uZ A E D  &          & \uZZ B A E \\
           & \uY B A &           & \uY C B &            & \uY D C &           & \uY E D &            & \uYY A E &            \\
\uXx A     &         & \uX B     &         & \uX C      &         & \uX D     &         & \uX E      &          & \uXX A
} $
\end{minipage}
\hspace{-0.1\textwidth} 
\begin{minipage}{0.375\textwidth}
\includegraphics[width=1.0\textwidth]{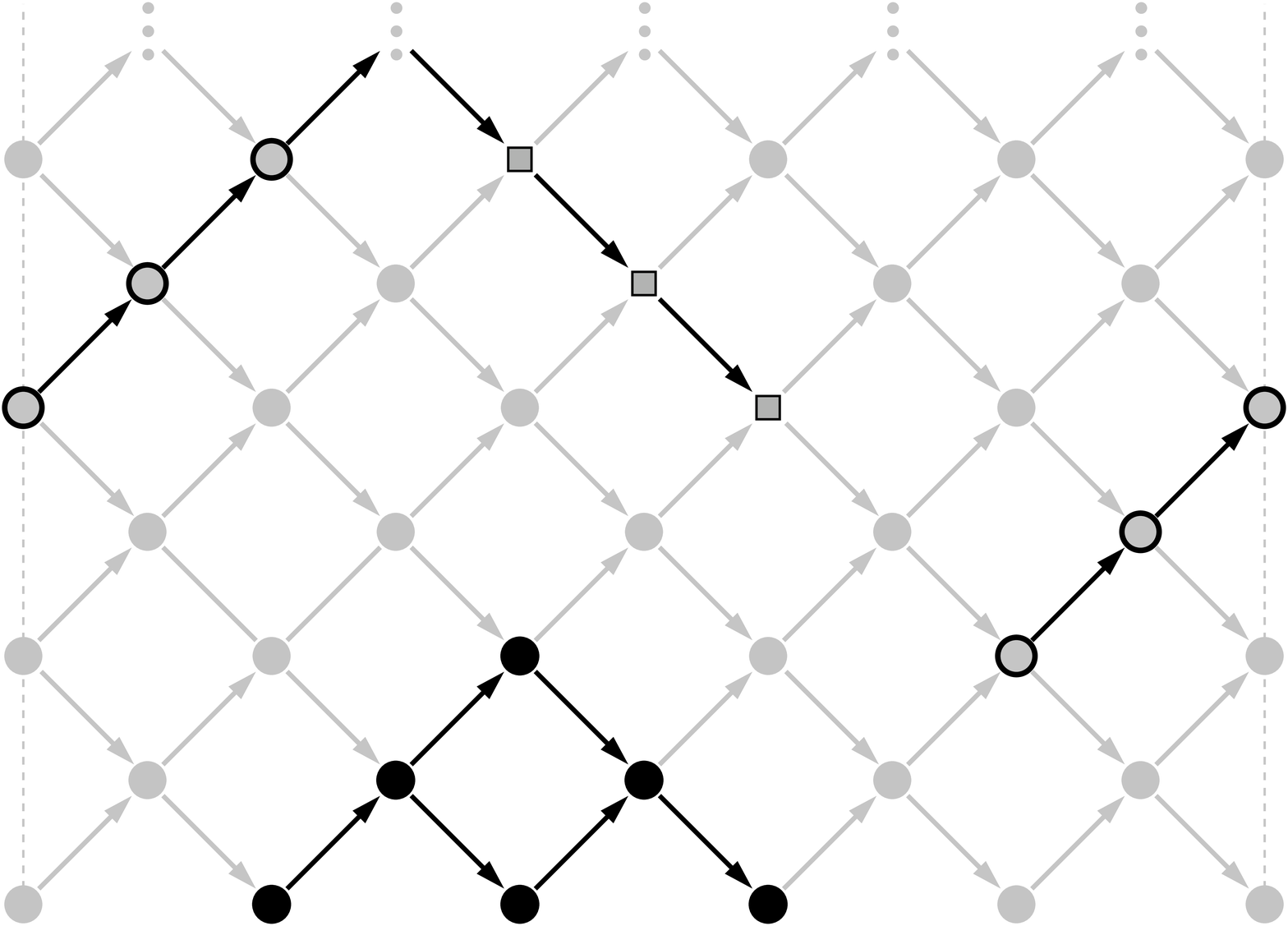} 
\end{minipage}
\caption[A standard stable tube]{ \label{fig:tube} \small
   A standard stable tube of rank 5. The outer columns are identified. \newline
   Left: composition series for indecomposable objects of length up to four. \newline
   Right: a wing (\dotspace\blackdot), part of a ray (\dotspace\rimdot), part of a coray (\dotspace\greyboxdot).
}
\end{figure}

In the case of the derived category $\Db(\Lambda)$ of a tame hereditary algebra $\Lambda$, an indecomposable object of a tube $\cT$ of rank $\tuberk$ of length $\geq\tuberk$ admits a non-zero morphism \emph{from} any object in the non-regular component immediately to the left of $\cT$ in the AR quiver of $\Db(\Lambda)$ and a non-zero morphism \emph{to} any object in the non-regular component immediately to the right of $\cT$ in the AR quiver of $\Db(\Lambda)$. A reference for this general structure can be found in \cite[Section VIII.2]{ASS} in the case of $\mod{\Lambda}$, the structure for $\Db(\Lambda)$ can be deduced from this and the hereditary condition.

\subsection{Homomorphisms and extensions in tubes}\label{sec:hom-ext-in-tubes}

In subsequent sections it will be important to understand homomorphisms and extensions between indecomposable objects sitting in some finite wing at the bottom of a tube and the remaining indecomposable objects in the tube. The uniseriality of tubes makes computation of the dimension of Hom and Ext spaces very easy.

Let $\cT$ be a standard stable tube of rank $\tuberk$. Let $t\in\cT$ be an indecomposable object of length, $l:=\len{t}<\tuberk$. We want to classify homomorphisms to and from $t$ and extensions starting and ending at $t$.

On homomorphisms: If the top of $t$ maps onto the bottom of $t'$ then there is a map $t \to t'$. In particular, this implies $\Hom(t,t')=0$ for $t' \in \sW(t) \setminus \sC(t)$. There are also no maps from $t$ to elements of $\sR(\tau t), \sR(\tau^2 t), \ldots, \sR(\tau^{\tuberk-l} t)$, where $\tau$ is the AR translation. This gives the half `Hom-hammock' shown in Figure~\ref{fig:hammocks}. The other half of the Hom-hammock consisting of maps to $t$ can be constructed similarly. Note that, since $l<\tuberk$, all non-zero Hom-spaces involving $t$ are one-dimensional. The Ext-hammocks can be computed from the Hom-hammocks using the Auslander-Reiten formula: $\Ext^1_{\Lambda}(M,N) \cong D\Hom_{\Lambda}(N,\tau M)$, where $D(\cdot)$ denotes the $\kk$-dual. See Figure~\ref{fig:hammocks} for an example.

\begin{figure} 
\centering
\includegraphics[width=0.47\textwidth]{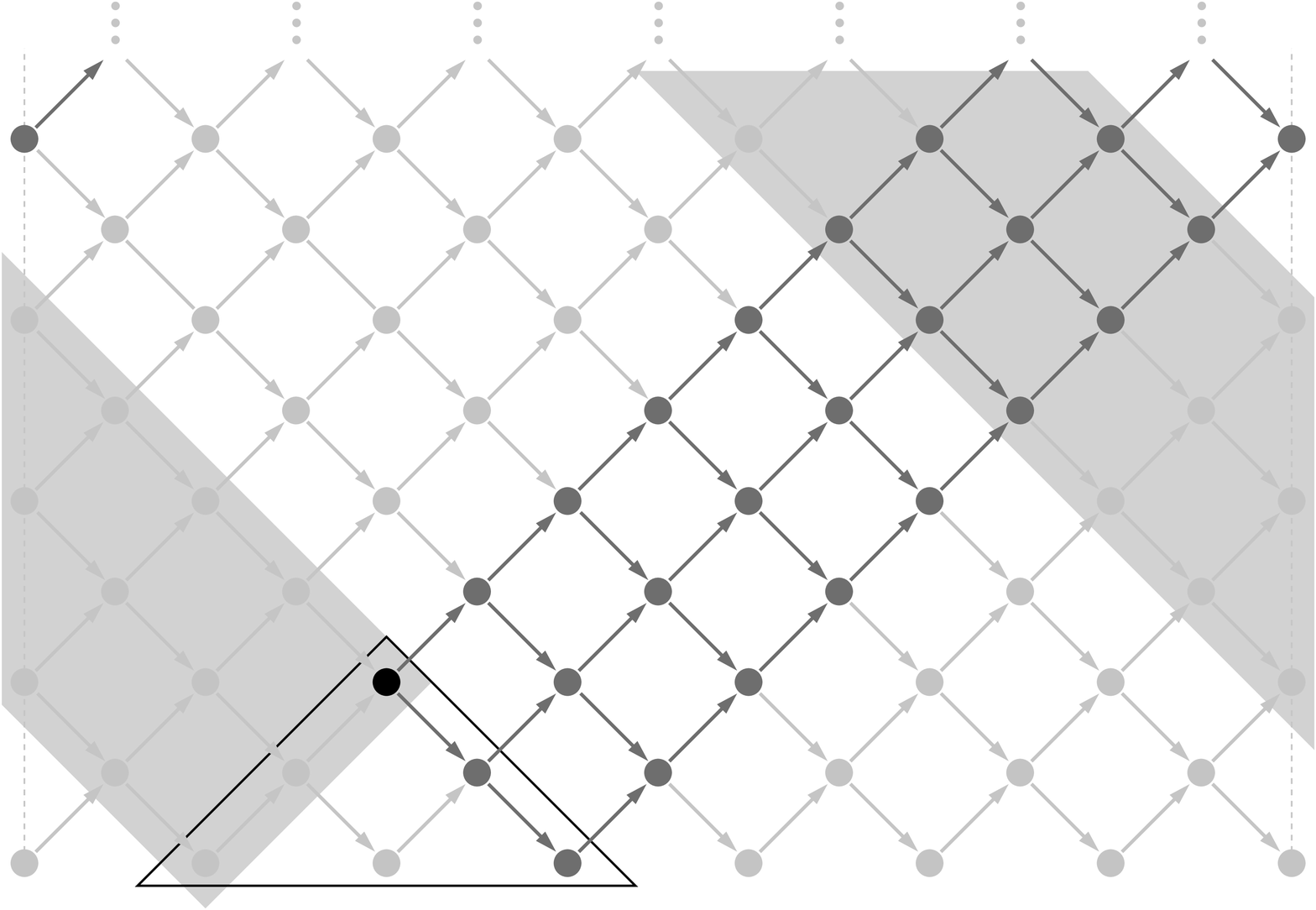} \hspace{0.04\textwidth}
\includegraphics[width=0.47\textwidth]{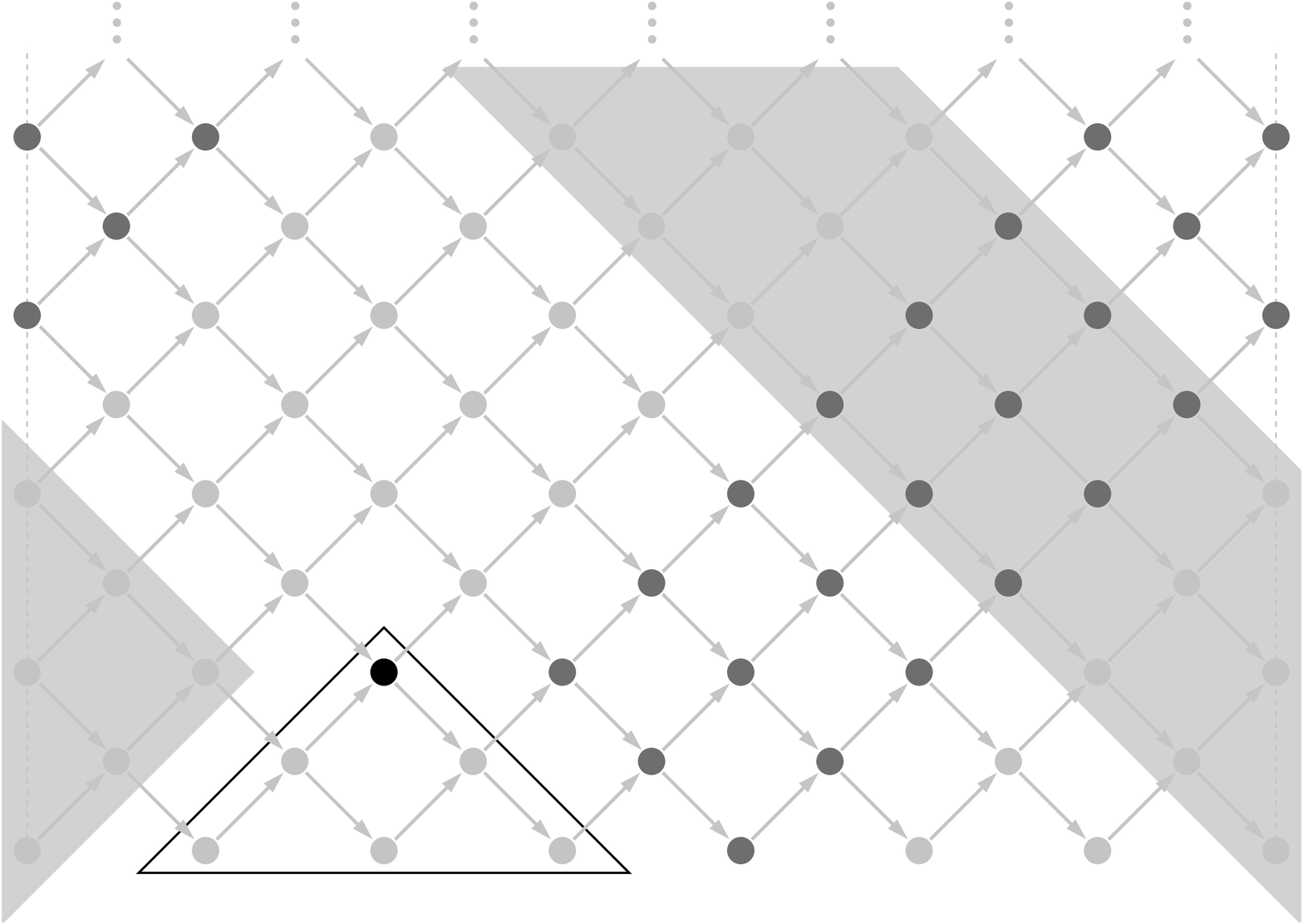}
\caption[Hom and Ext in tubes]{ \label{fig:hammocks} \small
   A tube of rank 7 and Hom/Ext vanishing for the object \dotspace\blackdot. \\
   $\Hom(\blackdot,\darkgreydot)\neq0$, $\Hom(\backgrounddot,\blackdot)\neq0$, \hfill
   $\Ext^1(\backgrounddot,\blackdot)\neq0$, $\Ext^1(\blackdot,\darkgreydot)\neq0$. \\
}
\end{figure}

It is useful to know what the indecomposable summands in an extension are. This can be done using the graphical calculus of \cite{Baur-Buan-Marsh}. In particular, given two objects $t,t'\in\cT$ (again with $t$ of length $l<\tuberk$), then an extension starting at $t'$ and ending at $t$ has the form 
$0 \to t' \to e_1 \oplus e_2 \to t \to 0$,
with $e_1$ and $e_2$ indecomposable objects, or zero. The objects $e_1$ and $e_2$ can be computed as follows: $e_1$ is the object in $\sC(t)\cap\sR(t')$ of shortest length, but longer than $t'$. Similarly, $e_2$ is the object in $\sR(t)\cap\sC(t')$ of longest length, but shorter than $t'$. We indicate an example computation in Figure~\ref{fig:tube-Lemma} on page \pageref{fig:tube-Lemma} in the proof of Lemma~\ref{lem:summands-of-v}.


\section{The Dynkin case and a piecewise tame hereditary example} \label{sec:dynkin}

\noindent
For finite representation type, the situation is as good as possible:

\begin{proposition} \label{prop:repfin}
Let $\sT$ be a piecewise hereditary triangulated category of finite type and $(\sX_i,\sY_i)_{i\in I}$ a finite set of t-structures on $\sT$. Then the refined truncation algorithm terminates for all objects of $\sT$.
\end{proposition}

\begin{proof}
Choose any indecomposable object $t\in\sT$. Without loss of generality, we may assume it lies in the heart $\sH$ of some standard t-structure in $\sT$. Suppose that the refined algorithm does not terminate. In this case, there exists a sequence $(t_n)$ of indecomposable summands of the $(\refinedY{t}{n})$, which is strictly increasing with respect to the partial order given by the AR quiver of $\sT$. Now, since the number of indecomposable objects in each degree is finite, $t_n\in\Sigma^m \sH$ for some $m\geq 2$ and $n$ large enough. Since $\Hom_{\sT}(\sH,\Sigma^m \sH)=0$ for all $m\geq 2$ by the hereditary property, we have that $\Hom_{\sT}(t,t_n)=0$. This contradicts the fact that $t_n$ is a summand of an object in the sequence $(\refinedY{t}{n})$. Hence, the refined algorithm terminates.
\end{proof}

Unfortunately, termination of the refined truncation algorithm does not always occur; this is shown in the following example.

\begin{example} \label{ex:no-t-structure}
Let $\sT = \Db(\kk\tilde{A}_{2,2})$, the bounded derived category of the path algebra of the extended Dynkin quiver $\tilde{A}_{2,2}$. This algebra is tame hereditary. It is also a canonical algebra and triangle equivalent to a geometric object, a stacky projective line with two points of isotropy of order 2. A concrete realisation is given by the quotient (Deligne-Mumford) stack $[\IP^1/G]$ where $G=\IZ/2\IZ$ acts with two fixed points, $0$ and $\infty$. Over $\kk=\IC$, one can think of the rotation of the sphere $S^2=\IP^1_\IC$ by the angle $\pi$, fixing the two poles. A well-known model for this stack is the `weighted projective line' $\IX(2,2)$ of Geigle and Lenzing \cite{Geigle-Lenzing}.
The abelian category $\coh(\IX(2,2))$, or rather its AR quiver, consists of a regular component $\cR$ to the right of a non-regular component $\cN$, where  $\cR$ is made up of standard stable tubes $\cT_\lambda$ indexed by $\lambda\in\IP^1$. All but the tubes $\cT_0$ and $\cT_\infty$ are homogeneous, i.e.\ have rank 1. The tubes $\cT_0$ and $\cT_\infty$ are both of rank 2. The AR quiver is depicted in Figure~\ref{fig:X(2,2)_AR-quiver}.

\begin{figure}
\center
\includegraphics[width=0.7\textwidth]{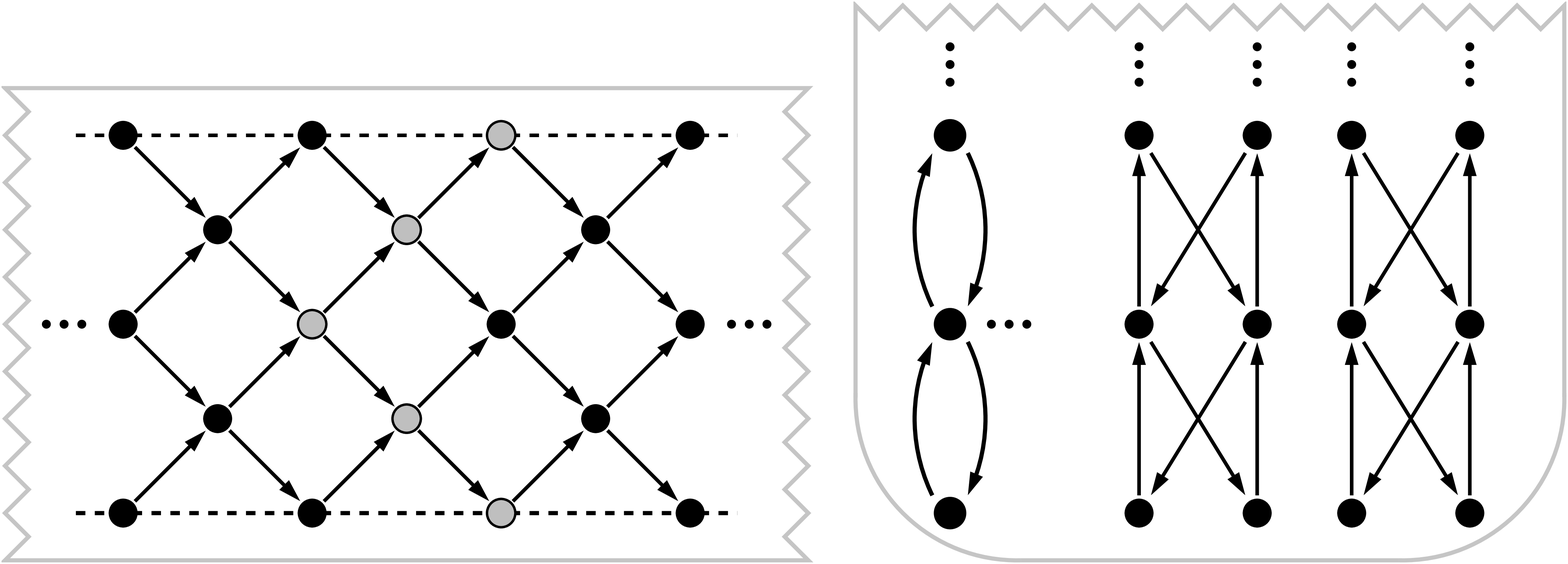}
\caption[The Auslander-Reiten quiver of $\coh(\IX(2,2))$.]{ \label{fig:X(2,2)_AR-quiver} \small
  Auslander-Reiten quiver of $\coh(\IX(2,2))$.
  The left-hand boxed area contains the non-regular component containing the line bundles;
  top and bottom rows are identified. The marked subquiver is of type $\tilde{A}_{2,2}$.
  The right-hand area is the regular component containing all torsion sheaves; it is made up of
  tubes which are parametrised by $\IP^1=\kk\cup\{\infty\}$. The tubes for $\lambda=0,\infty$ 
  have rank 2.
}
\end{figure}

Let $s_1$ and $s_2$ denote the quasi-simple objects at the mouths of the tube $\cT_0$.
We define two t-structures in $\sT$ as follows; one of them is shown in Figure~\ref{fig:A2,2_t-structure}:
\[
\sX_p :=  \{s_p\} \cup \Sigma\cR \cup \bigcup_{i\geq2} \Sigma^i \coh(\IX(2,2)), \qquad
\sY_p := \sX_p^\perp \qquad\text{for } p=1,2.
\]

\begin{figure}
\centering
\psfrag{a}{}    
\includegraphics[width=\textwidth]{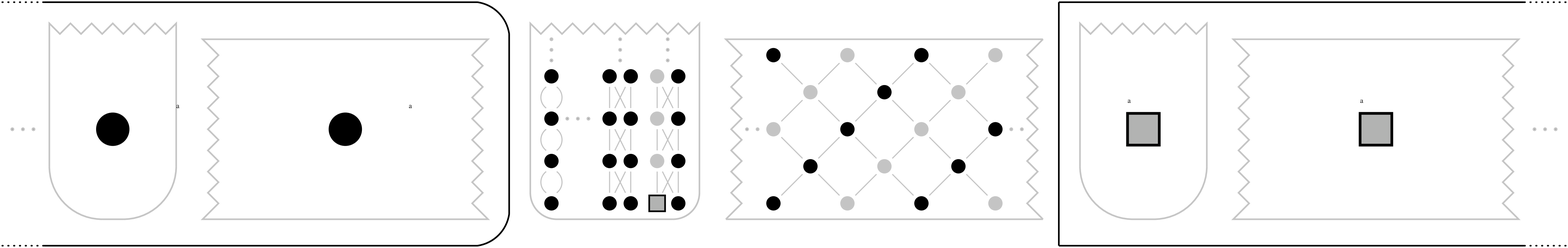}

$\hspace{0.20\textwidth} \cN       \hspace{0.15\textwidth} \cR
 \hspace{0.13\textwidth} \Sigma\cN \hspace{0.12\textwidth} \Sigma\cR
 \hfill$

\vspace{-2ex}
$\hspace{0.13\textwidth} \underbrace{\hspace{0.31\textwidth}}_{\coh(\IX(2,2))}
 \hspace{0.02\textwidth} \underbrace{\hspace{0.30\textwidth}}_{\Sigma\coh(\IX(2,2))}
 \hfill$
\caption[A t-structure on $\sT=\Db(\IX(2,2))$.]{ \label{fig:A2,2_t-structure} \small
  A t-structure on $\sT=\Db(\IX(2,2))$ defined by $\sX$ (\dotspace\boxeddot) and $\sY$ (\dotspace\blackdot);
  the remaining irreducible objects (\dotspace\greydot) belong to neither $\sX$ nor $\sY$.
}
\end{figure}

Observe that the refined algorithm doesn't terminate for this example: for instance, taking an indecomposable object $t\in\Sigma \cN$, one sees that the objects  $\refinedY{t}{n}$ simply move a finite distance to the right in the component $\Sigma \cN$ without ever terminating.
\end{example}

\begin{remark}
In Example~\ref{ex:no-t-structure}, one can compute that $\sY^I \cap \Sigma\cN = \sY_1 \cap \sY_2 \cap \Sigma\cN = 0$. Similarly, $\sX^I \cap \Sigma \cN=0$. As a consequence, no indecomposable object $t\in\Sigma\cN$ admits a truncation triangle $\tri{x}{t}{y}$ with $x\in \sX^I$ and $y\in \sY^I$, so that $(\sX^I,\sY^I)$ is not a t-structure.
\end{remark}

\begin{remark}
Example~\ref{ex:no-t-structure} gives a concrete example in which the refined algorithm does not terminate; indeed, in this case $\sod{\sX_1, \sX_2}$ is not an aisle in $\sT$. However, $\sY_1 \cap \sY_2$ is a co-aisle; so one may think that we have made the wrong choice for the averaged t-structure. 
However, one can modify this example in such a way that neither $\sod{\sX_1, \sX_2}$ is an aisle nor $\sY_1 \cap \sY_2$ a co-aisle.
\end{remark}


\section{The combinatorial criterion and Theorem~\ref{thm:small}} \label{sec:proof1}

\noindent
Before proving Theorem~\ref{thm:small}, we restate it more precisely below:

\begin{theorem} \label{thm:small:extended}
Let $\sT$ be a $\kk$-linear, piecewise tame domestic hereditary triangulated category and let $(\sX_i,\sY_i)_{i\in I}$ be a finite set of t-structures on $\sT$, where $I=\{0,\ldots,d-1\}$. Then the following are all equivalent:
\begin{enumerate}[(a)]
\item $\sX^I$ is an aisle, i.e.\ $(\sX^I,\sY^I)$ is a t-structure.
\item For any object $t$ of $\sT$, the refined algorithm terminates.
\item \label{itm:combicond} For any connected non-regular component $\cN$ of $\sT$ such that $\sY_i\cap\cN$ has an infinite strictly increasing sequence for all $i\in I$, there is a strictly increasing sequence in $\sY^I\cap\cN$.
\end{enumerate}
\end{theorem}

Given an object $t\in\sT$, we apply the refined algorithm to obtain a sequence
\begin{equation} \label{eq:refined-sequence}
t \to \refinedY{t}{0} \to \refinedY{t}{1} \to \refinedY{t}{2} \to \cdots \to \refinedY{t}{n} \to \cdots,
\end{equation}
which we refer to as the `refined truncation sequence' for $t$. 
We shall show that any connected component of the AR quiver of $\sT$ can contain only finitely many terms of this sequence.
The proof of this fact is not difficult, but rather technical, and broken down into sections. These are then put together to prove Theorem~\ref{thm:small:extended} in Section~\ref{sec:proof-small-extended}.

\subsection{Some technical lemmas} \label{sec:technical}

In the proof of Theorem~\ref{thm:small}, direct sum decompositions will be uniquitious. Here, we collect some nice properties of truncations regarding direct summands. These formal properties are surely well-known, and their proofs follow immediately from the definition of t-structure.

Let $\sT$ be a Krull-Schmidt $\kk$-linear triangulated category.
Suppose $(\sX,\sY)$ is a t-structure in $\sT$ and an object $t\in\sT$ has truncation triangle
\begin{equation} \label{eqn:trunc1}
x_1 \oplus x_2 \xrightarrow{[f_1\, f_2]} t \too y \too \Sigma (x_1 \oplus x_2)
\end{equation}
with respect to $(\sX,\sY)$, where $x_1\neq0$ and $x_2\neq0$.
In this sitation, the next three statements assert good behaviour of the right trunction maps, the octahedral diagram associated to the direct sum, and multiplicities.

\begin{lemma} \label{lem:non-zero-summands}
In the truncation triangle \eqref{eqn:trunc1}, the maps $f_1$ and $f_2$ are non-zero.
\end{lemma}

\begin{lemma} \label{lem:non-zero-maps}
The truncation triangle~\eqref{eqn:trunc1} fits into the following octahedral diagram:
\[ \xymatrix@C+1em{
x_1            \ar@{=}[r]           \ar[d]_-{\iota}  & x_1        \ar[d]^-{f_1}  &             \\
x_1 \oplus x_2 \ar[r]^-{[f_1\, f_2]} \ar[d]_-{\pi}    & t    \ar[r] \ar[d]^-{h}   & y \ar@{=}[d] \\
x_2 \ar[r]_-{g}                                                                   & u \ar[r]     & y	 
} \]
where $\iota$ and $\pi$ denote the canonical inclusion and projection maps. In this octahedron, the map $g\colon x_2\rightarrow u$ is non-zero.
\end{lemma}

\begin{lemma} \label{lem:summand-multiplicity}
If in the truncation triangle~\eqref{eqn:trunc1} $\dim \Hom_{\sT}(x_1,t)=1$ holds, then the multiplicity of $x_1$ as a summand of $x:=x_1\oplus x_2$ is one.
\end{lemma}

\subsection{Finiteness of the refined truncation sequence in a non-regular component} \label{sec:non-regular}

In this section we show the required finiteness of the refined truncation sequence under hypothesis~(\ref{itm:combicond}) of Theorem \ref{thm:small:extended}.

\begin{lemma} \label{lem:nonregint}
If $\cN$ is a connected non-regular component of $\sT$ such that $\sY^I\cap\cN$ contains an infinite strictly increasing sequence, then $ \sX^I \cap \cN = 0$
\end{lemma}

\begin{proof}
Suppose $x\in\ind \sX^I \cap \cN$. Then $x$ is a projective object in the heart $\sH$ of some standard t-structure in $\sT$. Now $\sH$ is a tame hereditary abelian category with postprojective component $\cP$ such that $\sY^I \cap \cP$ contains an infinite strictly increasing sequence. By \cite[Proposition IX.5.6]{ASS}, all but finitely many indecomposable objects in $\cP$ are sincere, i.e.\ admit non-zero maps from all projective modules. In particular, $x$ admits a non-zero morphism to all but finitely many objects of $\sY^I \cap \cP$, contradicting the left orthogonality of $\sX^I$ to $\sY^I$, hence $x=0$.
\end{proof}

\begin{lemma}\label{lem:nonregulartermination}
Let $t$ be an object whose indecomposable summands lie in a connected non-regular component $\cN$, and suppose condition~(\ref{itm:combicond}) of Theorem~\ref{thm:small:extended} holds. Then there exists $n_0>0$ such that $\add(\refinedY{t}{n}) \cap \cN \subset \sY^I$ for all $n >n_0$.
\end{lemma}

\begin{proof} We treat two cases:

\noindent\emph{Case 1. There is an $i \in I$ such that $\sY_i \cap \cN$ has no strictly increasing sequence of objects.}
Let ${\refinedY{t}{n}'}$ denote the sum of the summands of $\refinedY{t}{n}$ which lie in $\cN$. We obtain a sequence of morphisms
 $ t \to \refinedY{t}{0}' \to \refinedY{t}{1}' \to \refinedY{t}{2}' \to \refinedY{t}{3}' \to \cdots $
where the composition from $t$ to any summand of any ${\refinedY{t}{i}'}$ is non-zero. 
By the assumption of this case, the subsequence 
 $ \refinedY{t}{i}' \to \refinedY{t}{2i}' \to \refinedY{t}{3i}' \to \refinedY{t}{4i}' \to \cdots $
of objects in $\sY_i \cap \cN$ must stabilise, that is, $\refinedY{t}{ki}' \cong \refinedY{t}{(k+1)i}'$ for all $k\gg 0$. Furthermore, using the fact that the morphisms define a partial order we see that $\refinedY{t}{n}' \cong \refinedY{t}{n_0}' =: \tilde{y}$ for all $n \geq n_0$ and so $\tilde{y} \in \sY^I$ as required.

\noindent\emph{Case 2. For all $i \in I$ there exists a strictly increasing sequence of objects in $\sY_i \cap \cN$.} It follows from condition~(\ref{itm:combicond}) that $\sY^I \cap \cN$ has an infinite strictly increasing sequence of elements. Let $\cR$ denote the regular component with non-zero morphisms into $\cN$.

Let $\cT\subset\cR$ be a tube of rank $\tuberk$. Then there is a non-zero morphism from any element $t'\in\cT$ of length $\len{t'}\geq\tuberk$ to some element of any given infinite strictly increasing sequence in $\cN$; see Section~\ref{sec:structure-tubes}.
As $\sX^I \subset {}^{\perp} (\sY^I)$ and $\sX^I$ is extension closed, all indecomposable objects of $\sX^I \cap \cR $ are contained in wings of the inhomogeneous tubes of $\cR$. Lemma~\ref{lem:nonregint} implies that $\sX^I \cap \cN =0$. Putting these two statements together with the fact that $\Sigma \sX^I \subset \sX^I$, we see that all indecomposable objects of $\sX^I \cap \bigcup_{k \in \IN}\Sigma^{-k}(\cR \cup \cN)$ are contained in wings of the inhomogeneous tubes.

Now take any indecomposable object $t \in \cN$. For each $n \in \IN$, the refined algorithm yields
 $ \refinedX{t}{n} \to t \to \refinedY{t}{n} \to \Sigma \refinedX{t}{n} $
where $\refinedX{t}{n}\in\sX^I$.
Since there are only a finite number of (non-zero) morphisms from indecomposable objects in $\sX^I$ to $t$, up to repeated summands, there are only a finite number of cones of morphisms from $\sX^I$ to $t$. The result then follows from the partial order.
\end{proof}

Before obtaining the analogous finiteness statement for the refined truncation sequence in regular components, we first need to establish a bound on the length of indecomposable summands lying in a tube $\cT$ of the left truncation of a t-structure $(\sX,\sY)$ in $\sT$ when $\card{\ind (\sY \cap \cT)} = \infty$. This is done in the next subsection.

\subsection{A bound for the length of indecomposable summands in a tube} \label{sec:bound}

\begin{setup}\label{set:tube}
Throughout this section, $\cT$ will be a standard stable tube of rank $\tuberk>1$ in a piecewise tame hereditary triangulated category $\sT$. We consider a t-structure $(\sX,\sY)$ on $\sT$ with the property that $\card{\ind (\sY \cap \cT)} = \infty$. Using facts about the morphisms in $\cT$ discussed in Section~\ref{sec:hom-ext-in-tubes} this implies that all indecomposable objects in $\sX \cap \cT$ have length less than $\tuberk$.
We consider the truncation triangle of an indecomposable object $t\in \cT$ with respect to the t-structure $(\sX,\sY)$: 
\begin{equation}\label{eq:truncation}
    \tri{t_{\sX}}{t}{t_{\sY}}.
\end{equation}
Let $\cN$ be the unique non-regular component which admits non-zero morphisms to $\cT$. Then for any object of $\cN$, there is a non-zero morphism to an object in $\sY \cap \cT$ (see Section~\ref{sec:structure-tubes}). It follows that no summand of $t_{\sX}$ can lie in $\cN$ and therefore $t_{\sX}$ decomposes as
$t_{\sX} \cong \Sigma^{-1} r' \oplus r $,
with $r',r\in\cT$.
\end{setup}

\begin{lemma} \label{lem:nohom}
If $t\in \cT$ and $r, s$ are indecomposable summands of $t_\sX$, then $\Hom(r, s)=0$.
\end{lemma}
\begin{proof}
First recall from the setup that $r, s \in \Sigma^{-1} \cT \cup \cT$ with $\len{r}<\tuberk$ and $\len{s}<\tuberk$.
Take any morphism $\psi\colon r \to s$ and consider the composition $\varphi\colon r \to s \hookrightarrow t_\sX \to t$. Since $\len{r}<\tuberk$ we know that $\dim \Hom(r, t)=1$ and so if $\varphi \neq 0$ it is, up to a non-zero scalar multiple, this unique morphism. However $r\to t$ factors through $t_\sX \to t$ in two distinct ways (as a summand of $t_X$ and via $\phi$) which is a contradiction.
Therefore $\varphi$ is zero. Consider the diagram,
\[
\xymatrix{
 r \ar[r]^-{\psi}\ar@{-->}[dr]_-{\varphi = 0} & s \ar[r] \ar[d] & c \ar[r] \ar@{-->}[dl]^-{\exists} & \Sigma r , \\
                                              & t               &                                   &
}
\]
where $c$ is the cone of $\psi$. Since $\sX$ is closed under cones, $c$ is in $\sX$ and the morphism $c \to t$ must factor through the truncation $t_\sX$. By the uniqueness of such a factorisation, it must factor back through the summand $s$. Since $\len{s}<\tuberk$ we know that $\dim\Hom(s,s) = 1$, and so the morphism $s \to c \to s$ is the identity up to a scalar multiple. Therefore, the triangle splits and the morphism $\psi\colon r \to s$ must be zero.
\end{proof}

The following proposition establishes the bound; its proof is a sequence of lemmas.

\begin{proposition} \label{prop:length-bound}
Under the conditions of Setup~\ref{set:tube}, any indecomposable summand of $t_{\sY}$ lying in $\cT$ has length at most 
 $ \tuberk \ceil{\len{t}/\tuberk} $.
\end{proposition}

\begin{proof}
Suppose that we have a direct sum decomposition of an object $a= \bigoplus_{i=1}^l a_i$, and a morphism $a \to t$.  Using octahedron diagrams like the following,
\begin{equation} \label{eq:main-octahedron}
\xymatrix{
\bigoplus_{i=1}^k a_i       \ar@{=}[r]        \ar[d]_-{\iota}  & \bigoplus_{i=1}^k a_i \ar[d]^-{f}     &                           \\
\bigoplus_{i=1}^{k+1} a_i  \ar[r]^-{ } \ar[d]_-{\pi}    & t 	            \ar[d] \ar[r]  & b_{k+1} \ar@{=}[d] 		\\
           a_{k+1}		\ar[r]^-{\phi}                     &  b_k                    \ar[r]  &  b_{k+1}
} \end{equation}
we see that the cone $b_l = \Cone(a \to t)$ can be calculated inductively, where $b_{k+1} = \Cone(a_{k+1} \to b_k)$ and $b_1 = \Cone(a_1 \to t)$. 
Note that although a choice is being made in the order of the summands, this is not important because the cone we construct is unique up to non-unique isomorphism.

Recall the decomposition of the right truncation $t_{\sX}$ into two summands given in Setup~\ref{set:tube}. We further decompose each of these into indecomposables and use the method outlined above to build up the cone $y$. We start by considering the summands of $\Sigma^{-1} r'$.

\begin{lemma} \label{lem:summands-of-v}
The object $v := \Cone( \Sigma^{-1} r' \to t)$ decomposes into direct summands, 
\[
v \cong v'_1 \oplus v'_2 \oplus \cdots \oplus v'_{m-1} \oplus v'_m \oplus v''_m ,
\]
with $\len{v'_1} > \len{v'_2} > \cdots > \len{v'_m} > \len{v''_m}$, where $m$ is the number of non-isomorphic indecomposable summands of $r'$. Furthermore, $v'_1$ is the only summand whose length is possibly greater than $\tuberk$.
\end{lemma}

\begin{proof}
Write $r'$ as a direct sum of indecomposables, $r'\cong r'_1 \oplus r'_2 \oplus \cdots \oplus r'_m$. It is clear from Lemma~\ref{lem:nohom} that each summand $r'_i$ lies on a distinct coray $\sC(r'_i)$. Let $s_i\in \IN$ be such that $\tau^{-s_i}t\in\sC(r'_i)$. Without loss of generality we relabel the summands such that $s_1 > s_2 > \cdots > s_m$. We know several facts about the summands $r'_i$ which severely restrict where they can lie in the tube:
\begin{itemize}
\item Each $r'_i$ is in $\sX$ and so $\len{r'_i} < \tuberk$.
\item There exists a non-zero morphism $\Sigma^{-1} r'_i \to t$, i.e.\ $r'_i$ lies in the Ext-hammock of $t$.
\item Lemma~\ref{lem:nohom} implies that $r'_{i+1}$ lies outside the Hom-hammock of $r'_i$.
\end{itemize}
These facts, and the ordering of the $\sC(r'_1),\ldots,\sC(r'_m)$, yield $r'_{i+1} \in \sW(r'_i)\backslash(\sR(r'_i) \cup \sC(r'_i))$. 

We now build the cone $v$ inductively. The object $v_1 = \Cone(\Sigma^{-1} r'_1 \to t)$ can be decomposed as $v_1=v'_1 \oplus v''_1$ with $v'_1 \in \sC(r'_1) \cap \sR(t)$  and $v''_1 \in \sR(r'_1) \cap \sC(t)$ (see Section~\ref{sec:background}).
Observe that $\len{v''_1} < \len{r'_1} < \tuberk$. Additionally, since $ \sC(v'_1) = \sC(r'_1)$ and $\len{r'_1}<\tuberk$, looking at the Ext-hammock of $v'_1$, we see that there are no non-split extensions from $v'_1$ to $\sW(r'_1)$. In particular $\Hom(\Sigma^{-1} r'_i,v'_1)=0$ for all $i=1,\dots,m$.

\bigskip \noindent 
\textbf{Induction step.} 
\emph{Let $j\geq 1$ and suppose that $v_j \cong v'_1 \oplus \cdots \oplus v'_j \oplus v''_j $ as a direct sum of indecomposable objects  with $\len{v'_1} > \cdots > \len{v'_j} > \len{v''_j}$ and only $\len{v'_1}$ possibly larger than $\tuberk-1$. Moreover, assume that $\Hom(\Sigma^{-1} r'_k, v'_i)=0$  for $i=1,\ldots,j$ and $k=j,\ldots,m$.}

\begin{center}
\psfrag{t}{$\scriptstyle t$}
\psfrag{r1}{$\scriptstyle r_1'$}
\psfrag{r2}{$\scriptstyle r_2'$}
\psfrag{r3}{$\scriptstyle r_3'$}
\psfrag{v1}{$\scriptstyle v_1'$}
\psfrag{v2}{$\scriptstyle v_2'$}
\psfrag{v3}{$\scriptstyle v_3'$}
\psfrag{w1}{$\scriptstyle v_1''$}
\psfrag{w2}{$\scriptstyle v_2''$}
\psfrag{w3}{$\scriptstyle v_3''$}
\psfrag{W}{$\sW(r_1')$}
\includegraphics[width=0.5\textwidth]{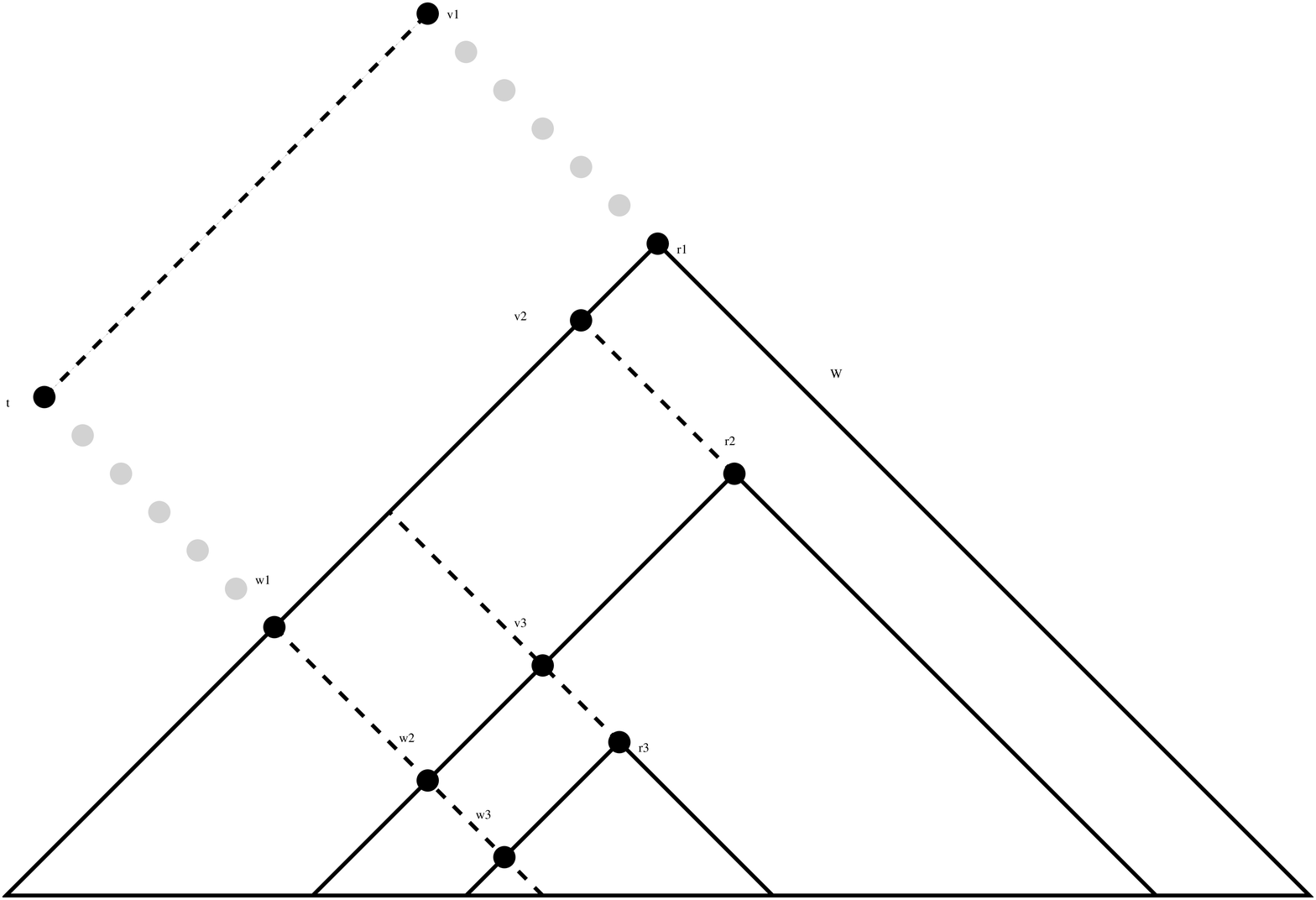}
\end{center}

\bigskip \noindent
The morphism $\Sigma^{-1}r'_{j+1} \to v_j$ is non-zero by Lemma~\ref{lem:non-zero-maps}, but by the inductive hypotheses there are no maps from $\Sigma^{-1}r'_{j+1}$ to any of the summands of $v_j$ except $v''_j$. Using the graphical calculus (see Section~\ref{sec:background}), we see that the cone of the morphism $\Sigma^{-1}r'_{j+1}\to v''_j$ decomposes as $v'_{j+1} \oplus v''_{j+1}$ with $v'_{j+1} \in \sC(r'_{j+1}) \cap \sR(v''_j) = \sC(r'_{j+1}) \cap \sR(r'_j)$ and $v''_{j+1}\in \sR(r'_{j+1}) \cap \sC(v''_j)$. Since $r'_{j+1} \in \sW(r'_j)$ we see that $\len{v'_j} > \len{v'_{j+1}} > \len{v''_{j+1}}$ and as $ \sC(v'_{j+1}) = \sC(r'_{j+1})$ and $\len{r'_{j+1}}<\tuberk$ we see from the Ext-hammock of $v'_{j+1}$ that there are no non-split extensions from $v'_{j+1}$ to $\sW(r'_{j+1})$. In particular $\Hom(\Sigma^{-1} r'_k,v'_{j+1})=0$ for all $k=j+1,\dots,m$. Finally we have the following diagram:
\begin{equation*} \label{eq:split-nonsplit}
\xymatrix{
                                           & \bigoplus_{i=1}^j \Sigma^{-1} v'_i \ar@{=}[r] \ar[d]_-{0} & \bigoplus_{i=1}^j \Sigma^{-1} v'_i \ar[d] \\
\Sigma^{-1} r'_{j+1} \ar[r]^-{ } \ar@{=}[d] & v''_j                             \ar[r]     \ar[d]      & v'_{j+1} \oplus v''_{j+1}          \ar[d] \\
\Sigma^{-1} r'_{j+1} \ar[r]                 &  v_j                              \ar[r]                 & v_{j+1}
} \end{equation*}

\noindent
from which we see that
 $ v_{j+1} = v'_1 \oplus \cdots \oplus v'_j \oplus v'_{j+1} \oplus v''_{j+1} $,
as required. This completes the proof of Lemma~\ref{lem:summands-of-v}, since $v=v_m$ by construction.
\end{proof}

We need to understand the length of the longest summand $v'_1$. The two typical situations are depicted in Figure~\ref{fig:tube-Lemma}.

\begin{figure}
\psfrag{A}{$A$}
\psfrag{t}{$\scriptstyle t$}
\psfrag{r1'}{$\scriptstyle r_1'$}
\psfrag{v1'}{$\scriptstyle v_1'$}
\psfrag{t2}{$\scriptstyle \tilde t$}
\psfrag{v2'}{$\scriptstyle \tilde v_1'$}
\psfrag{l1}{$\scriptstyle \text{length } 1$}
\psfrag{lR}{$\scriptstyle \text{length } \tuberk$}
\psfrag{l2R}{$\scriptstyle \text{length } 2\tuberk$}
\psfrag{lt}{$\scriptstyle \ell(t)$}
\psfrag{lt2}{$\scriptstyle \ell(\tilde t)$}
\psfrag{lv}{$\scriptstyle \ell(v_1')$}
\psfrag{lv2}{$\scriptstyle \ell(\tilde v_1')$}
\centering
\includegraphics[width=0.4\textwidth]{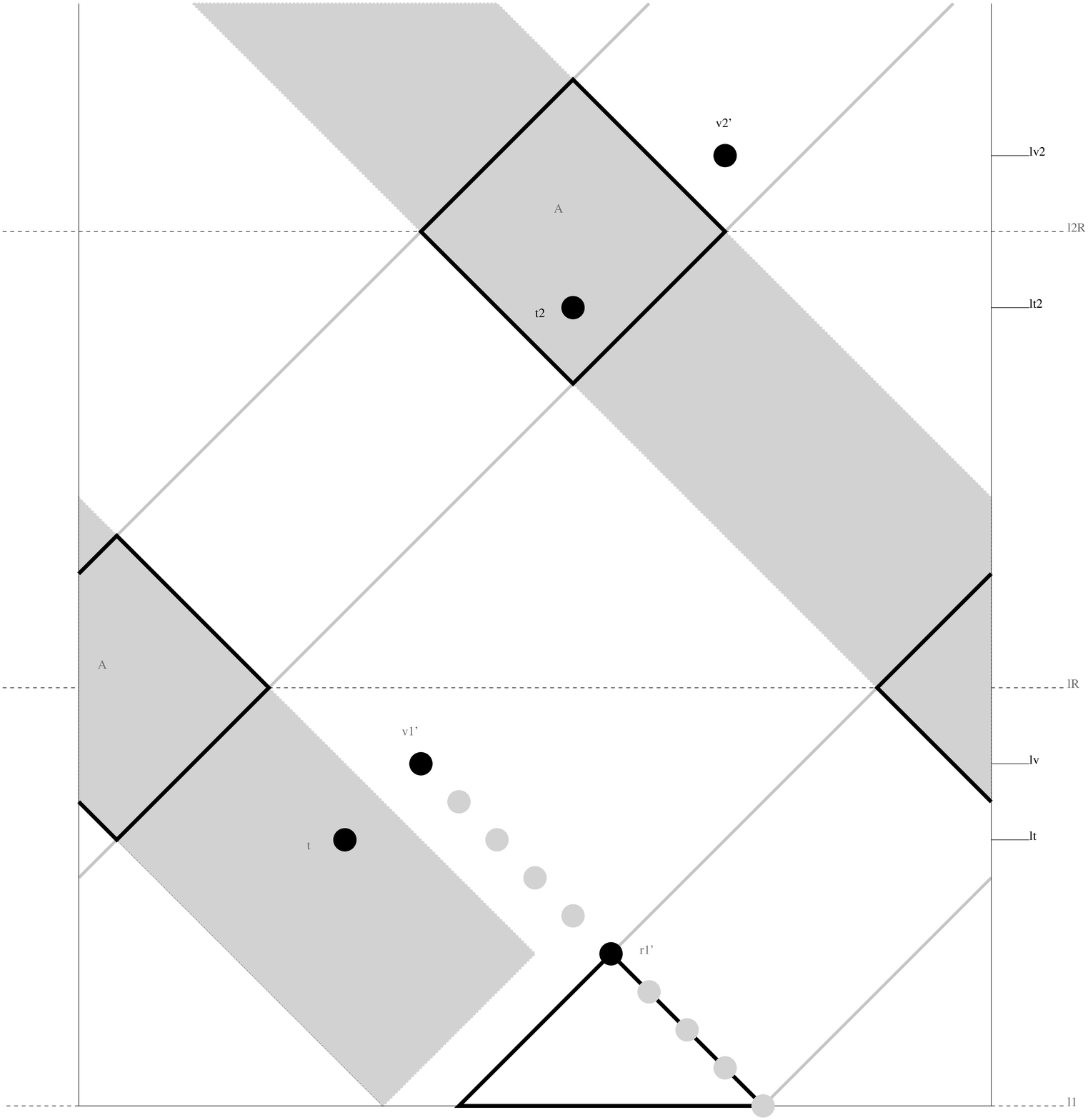} 
\caption[]{ \label{fig:tube-Lemma} \small
The shaded region is the $\Ext^1(r_1',\cdot)\neq0$ hammock. The squares marked $A$ are the intersection with the region $\Hom(r_1',\cdot)\neq0$. Objects $t$ and $\tilde t$ (which must be in the shaded region) show the two possibilities: $t$ and $v_1'$ are in the same $\tuberk$-region whereas $\tilde t$ and $\tilde v_1'$ are not (because $\tilde t$ is in the lower half of $A$).
}
\end{figure}

\begin{lemma} \label{lem:length-of-v}
With the notation from Lemma~\ref{lem:summands-of-v}, we have:
\begin{enumerate}
\item[(i)]  If there are no non-zero maps $r'_1 \to t$, then $\len{v'_1} \leq \tuberk \ceil{\len{t}/\tuberk}$.
\item[(ii)] If there is a non-zero map $r'_1 \to t$, then there exists an integer $0\leq k < \len{r'_1}$ such that $t\in\sR(\tau^{-k}r'_1)$ and, setting $l := \len{t} \bmod \tuberk$,
\[
\len{v'_1} =
\left\{ \begin{array}{ll}
           \tuberk \floor{\len{t}/\tuberk} + \len{r'_1} - k         & \text{if } 0 \leq l < \len{r'_1}, \\
           \tuberk \ceil{\len{t}/\tuberk}  + \len{r'_1} - k         & \text{if } \tuberk - \len{r'_1} < l < \tuberk.
\end{array} \right.
\]
\end{enumerate}
\end{lemma}

\begin{proof}
First note that in statement $(i)$, $\len{t}$ is not a multiple of $\tuberk$. Since there is no non-zero map $r'_1 \to t$, then 
$
t\in \sR(\tau r'_1) \cup \cdots \cup \sR(\tau^{\tuberk-\len{r'_1}} r'_1).
$
If $t\in\sR(\tau^i r'_1)$, then using the graphical calculus described in Section~\ref{sec:background}, we see that
\[
\len{v'_1} = \tuberk \lfloor \len{t}/\tuberk \rfloor + \len{r'_1} + i \leq \tuberk \lceil \len{t}/\tuberk \rceil
\]
since $\len{r'_1} +i \leq \tuberk$, as required.

For statement $(ii)$, since $\Hom(r'_1,t)\neq0$, the existence of the integer $k$ follows from Section~\ref{sec:background}. It is a straightforward exercise, using the graphical calculus, to then show the length (see Figure~\ref{fig:tube-Lemma}).
\end{proof}

Now we continue building $t_{\sY}$ by considering the indecomposable summands of $r$.

\begin{lemma}\label{lem:r-summands}
Let $s$ be an indecomposable summand of $r$, then \[\Hom(s, v'_2 \oplus \cdots \oplus v'_{m-1} \oplus v'_m \oplus v''_m)=0. \]
\end{lemma}

\begin{proof}
Since $\card{\ind ( \sY \cap\cT )} = \infty$, there exists some finite wing $\sW$ in $\cT$ such that $\ind (\sX \cap \cT) \subseteq \sW$. The summand $s$ must be an object in this wing. We also note that $\sW(r'_1)$ is a subwing of $\sW$. Consider the left-hand figure:

\noindent
\psfrag{W}{$\sW$}
\psfrag{W(r1')}{$\sW(r_1')$}
\psfrag{A}{$A_0$}
\psfrag{B}{$B_0$}
\psfrag{C}{$C_0$}
\psfrag{r}{$\scriptstyle r_1'$}
\scalebox{0.95}{\includegraphics[width=0.5\textwidth]{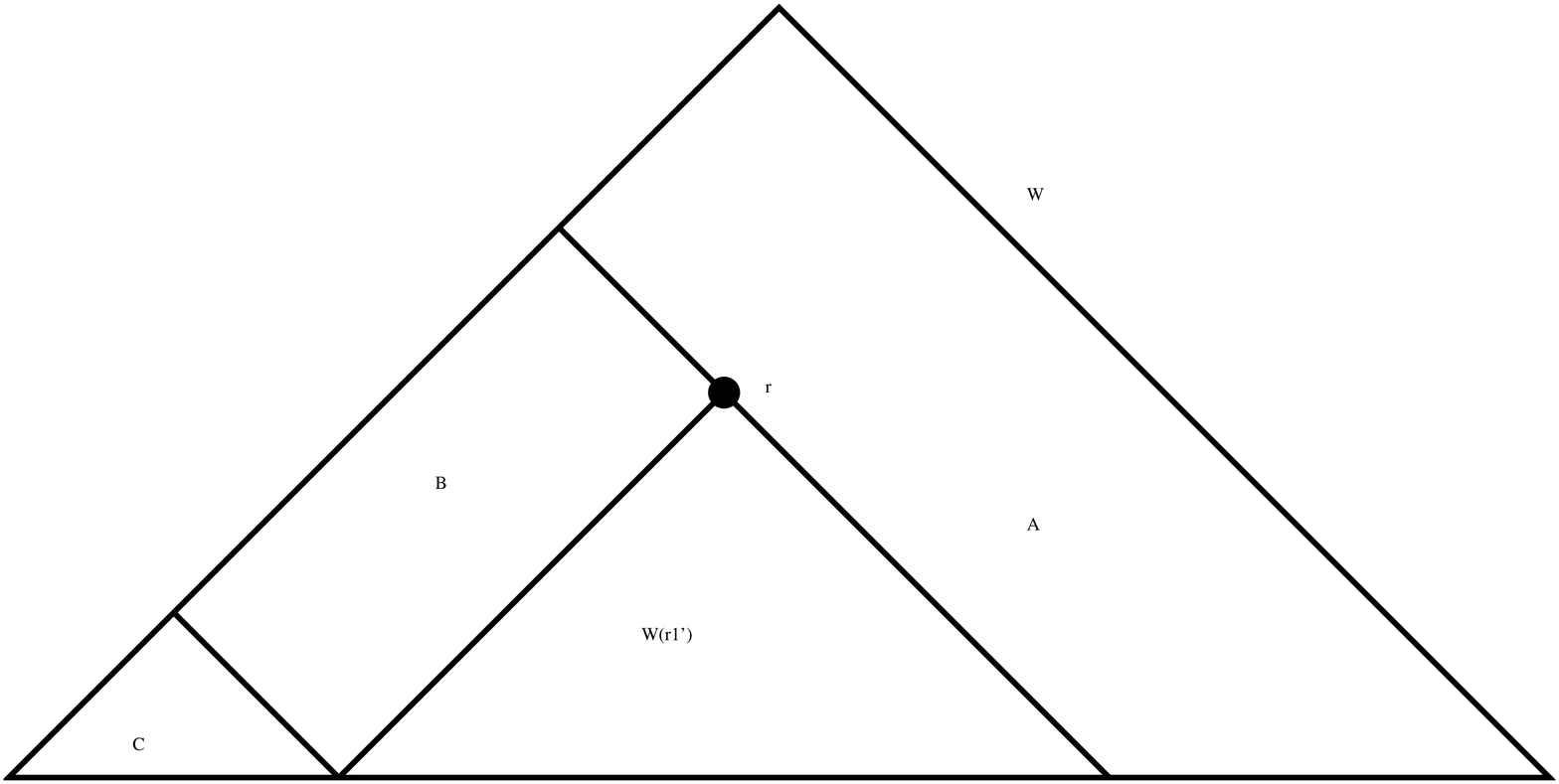}}
\hfill
\psfrag{W}{$\sW(r_i')$}
\psfrag{W(r1')}{$\!\!\sW(r_{i+1}')$}
\psfrag{A}{$A_i$}
\psfrag{B}{$B_i$}
\psfrag{C}{$C_i$}
\psfrag{r}{$\scriptstyle r_{i+1}'$}
\psfrag{rr}{$\scriptstyle r_i'$}
\scalebox{0.95}{\includegraphics[width=0.5\textwidth]{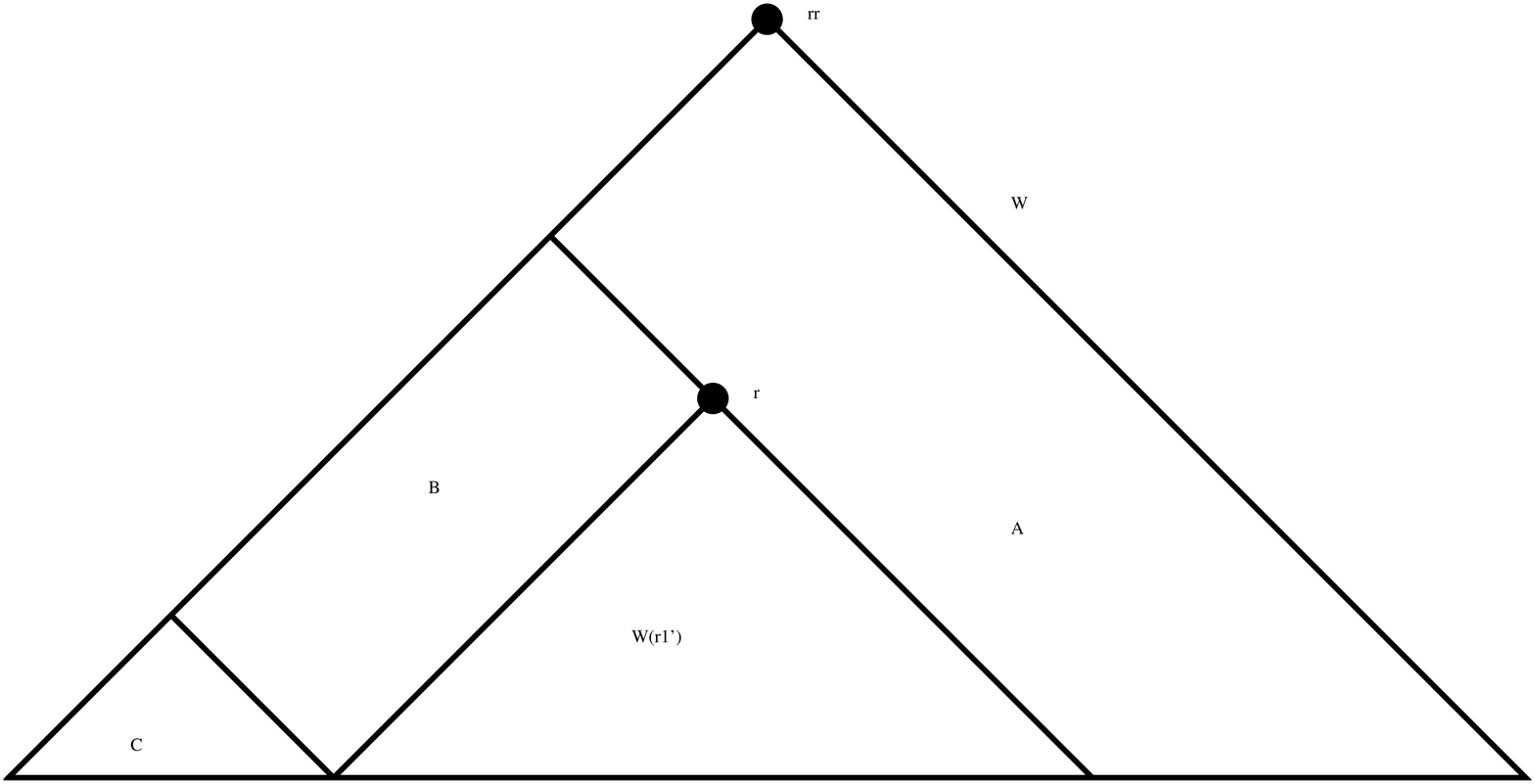}}

By the proof of Lemma~\ref{lem:summands-of-v}, $v'_2, \dots v'_m, v''_m \in \sW(r'_1)\backslash \sC(r'_1)$. Therefore, looking at the Hom-hammock, if $s$ is in the (closed) region marked $A_0$, we are done. Similarly, if $s \in C_0 \backslash (C_0 \cap B_0)$ the result holds. If $s$ were in $B_0 \backslash (\sR(r'_1) \cup \sC(r'_1))$ then by looking at the Ext-hammock of $s$ one can see that there would exist a non-zero morphism $\Sigma^{-1} r'_1 \to s$, but this cannot happen by Lemma~\ref{lem:nohom}. Therefore we have shown that the lemma holds for all possible $s$ outside $\sW(r'_1)$.

We proceed inductively: Suppose that the lemma holds for all possible $s$ outside $\sW(r'_i)$ for some $i \in \{1, \dots , m-1 \}$. 
First we observe (see right-hand figure above) that $\Hom(\sW(r'_i), v'_j) =0$ for $j=2, \dots , i$. If $s \in A_i \backslash \sR(r'_i)$ or $s \in C_i \backslash (B_i \cup \sR(r'_i))$, then in addition $\Hom(s, v'_{i+1})=0$ and $\Hom(s, \sW(r'_{i+1}))=0$. Since $v'_{i+2}, \dots, v'_m, v''_m \in \sW(r'_{i+1})\backslash \sC(r'_1)$, the result holds. For $B_i \backslash (\sR(r'_{i+1}) \cup \sC(r'_{i+1}))$ the same argument as above, replacing $r'_1$ by $r'_{i+1}$, shows that the summand $s$ doesn't lie in this region. Therefore in order to complete the inductive step we just need to consider $s \in \sR(r'_i) \cap \sW(r'_i)$. In this case, looking at the Hom-hammocks, we see that the morphism $s \to t$ must factor through a morphism $r'_i \to t$, which in turn must factor through $t_\sX$. This would contradict Lemma~\ref{lem:nohom} unless $s=r'_i$, which was a case already considered in region $A_{i-1}$. At the $m$-th step, the region $A_m$ is the whole wing $\sW(r'_m)$, and after this all possibilities for $s$ have been exhausted.  
\end{proof}

We observe that if $\Hom(a,c)=0$ then $ \Cone(a \to b \oplus c) = \Cone(a \to b) \oplus c$. Therefore, as we successively take cones from the summands of $r$ to build the truncation $y$, we see that $v'_2 \oplus \cdots \oplus v'_{m-1} \oplus v'_m \oplus v''_m$ remains a fixed summand at each step, and the remaining summands are obtained from $v'_1$ by successively taking cones from the summands of $r$. Since $v'_2, \dots , v'_{m-1} , v'_m , v''_m$ all have length less than $\tuberk$, they don't affect the bound we are trying to prove, so from now on we will only consider those summands obtained from $v'_1$. 

\begin{lemma} \label{lem:summands-of-w}
The object $w := \Cone( r \to v'_1)$ decomposes into a sum of direct summands, 
\[
w \cong w'_{p} \oplus \Sigma w''_1 \oplus \cdots \oplus \Sigma w''_{p}  ,
\]
where $w''_1, \dots, w''_{p}, w'_{p}\in\ind(\cT)$ and $p$ is the number of non-isomorphic indecomposable summands of $r$. The only summand whose length is possibly greater than $\tuberk$ is $w'_{p}$ which can be calculated inductively by $w'_j = \coker(r_j \to w'_{j-1})$, where $w'_0 := v'_1$.
\end{lemma}

\begin{proof}
The proof is similar to that of Lemma~\ref{lem:summands-of-v}. Suppose $r$ decomposes into indecomposables $r = r_1\oplus\cdots\oplus r_p$. It is clear from Lemma~\ref{lem:nohom} that each summand $r_i$ lies on a distinct ray $\sR(r_i)$ and let $s_i\in \IN$ be such that $\tau^{-s_i}t\in\sR(r_i)$. Without loss of generality we relabel the summands such that $s_1 > s_2 > \cdots > s_p$. The cone $w_1$ of the morphism $r_1 \to v'_1$ splits as the sum of the cokernel and the shifted kernel $w_1= w'_1 \oplus \Sigma w''_1$, where $w'_1 \in \sC(v'_1)$ with $\len{w'_1}<\len{v'_1}$ and $w''_1 \in \sR(r_1)$ with $\len{w''_1}<\len{r_1}$. Using the ordering of the $r_i$, we see that $\Ext^1(r_i,w''_1)=0$ for $i>1$.

As before we proceed inductively. Suppose $w_j:= \Cone(r_j \to w_{j-1})$ decomposes as $w_j= w'_j \oplus \Sigma w''_1 \oplus \cdots \oplus \Sigma w''_j $ such that $\Ext^1(r_i, w''_k)=0$ for all $k \leq j$ and $i>k$. Then, using the extension vanishing one can show that $w_{j+1}:= \Cone(r_{j+1} \to w_j)$ decomposes as $w_{j+1}= w'_{j+1} \oplus \Sigma w''_1 \oplus \cdots \oplus \Sigma w''_{j+1} $ where $w'_{j+1}$ and $w''_{j+1}$ are respectively the cokernel and kernel of the morphism $r_{j+1} \to w'_j$. Since $\sR(w''_{j+1}) = \sR(r_{j+1})$, one can see from the ordering of the $r_i$ that $\Ext^1(r_i, w''_{j+1})=0$ for all $i > j+1$. This completes the induction.
\end{proof}

\noindent
Finally, we show that the summand $w'_{p}$ satisfies the bound from Proposition~\ref{prop:length-bound}.

\begin{lemma}
Let $w'_{p}$ be the summand defined above. Then $\len{w'_{p}} \leq \tuberk \lCeil \dfrac{\len{t}}{\tuberk} \rCeil.$
\end{lemma}

\begin{proof}
From $w'_j = \coker(r_j \to w'_{j-1})$ we get $w'_j\in\sC(v'_1)$ and $\len{v'_1} \geq \len{w'_1} \geq \dots \geq \len{w'_{p}}$. If $\Hom(r'_1,t)=0$, then $\len{w'_{p}} \leq \len{v'_1} \leq \tuberk \ceil{\len{t}/\tuberk}$ by Lemma~\ref{lem:length-of-v}.

Any morphism $r'_1 \to t$ must factor through $t_\sX$ and therefore, through some summands of $r$. Let $r_i$ be one of these summands.
This must lie in the intersection of the $\sX$-wing (the smallest wing containing $\sX$), the Hom-hammock from $r'_1 $ and the Hom-hammock to $t$. There is a unique morphism from each $\tilde{r}$ in this region to $v'_1$, since $\sR(v'_1) = \sR(t)$ and $\len{\tilde{r}}<\tuberk$. As the morphisms going up the rays are monomorphisms, it is clear that the length of the cokernels of these morphisms decrease as $\tilde{r}$ goes up a ray. It is also straightforward to check that for $\tilde{r}$ in this region, the length of the cokernel only depends on the ray (where it lies on the coray affects the length of the kernel, not the cokernel). Therefore, the maximal length of $\coker(\tilde{r} \to v'_1)$ occurs when $\tilde{r} \in \sC(r'_1) \cap \sR(t)$; see $\tilde t$ in Figure~\ref{fig:tube-Lemma}.
In this case $\tilde{r}$ has length $\len{r'_1}-k $,  where $k$ is the integer introduced in Lemma~\ref{lem:length-of-v} and the morphism $\tilde{r} \to v'_1$ is a monomorphism. It follows that 
\[
\len{\coker(r_i \to v'_1)} \leq \len{\coker(\tilde{r} \to v'_1)} = \len{v'_1} - (\len{r'_1}-k) \leq \tuberk \ceil{\len{t}/\tuberk}
\] 
where the last inequality comes from Lemma~\ref{lem:length-of-v}.
Thinking in terms of the composition series, calculating the cokernels $w'_j = \coker(r_j \to w'_{j-1})$ involves successively removing some part from the bottom of the composition series, starting with the series for $v'_1$. If the composition $r_i \to v'_1  \to w'_{i-1}$ is zero, then the bottom part of the composition series of $v'_1$ into which $r_i$ maps has already gone from $w'_{i-1}$.  If some of it remains, so $r_i \to v'_1  \to w'_{i-1}$ is non-zero, then this part of the series is removed when we take the cokernel. In particular, $ \len{w'_i} \leq \len{\coker(r_i \to v'_1)}$.
\end{proof}

\noindent
This lemma completes the proof of Proposition~\ref{prop:length-bound}. 
\end{proof}

\subsection{Finiteness of the refined truncation sequence in a regular component}\label{sec:regular}

Recall from Section~\ref{sec:background} that the regular component consists of an infinite number of standard stable homogeneous tubes (i.e.\ rank one), and finitely many standard stable tubes of ranks larger than one.

\begin{proposition}\label{prop:regulartermination}
Let $t$ be an indecomposable object in a standard stable tube $\cT$ of rank $\tuberk>1$. Consider the sequence of morphisms obtained by applying the refined algorithm:
\[ t \too \refinedY{t}{1} \too \refinedY{t}{2} \too \refinedY{t}{3} \too \cdots, \]
then the projection onto $\cT$ yields a sequence
\begin{equation}\label{eq:main-sequence}
t \too \refinedY{t}{1}' \too \refinedY{t}{2}' \too \refinedY{t}{3}' \too \cdots,
\end{equation}
which stabilises after finitely many iterations of the refined algorithm, where $\refinedY{t}{i}'$ denotes the summand of $\refinedY{t}{i}$ lying in $\cT$ for each $i$. 
\end{proposition}

\begin{proof}
Suppose there exists $i \in \{0,\ldots,d-1\}$ such that $\card{\ind (\sY_i \cap \cT)} < \infty$. 
We consider the subsequence
 $ t \to \refinedY{t}{i}' \to \refinedY{t}{i+d}' \to \refinedY{t}{i+2d}' \to \cdots $
of sequence \eqref{eq:main-sequence} of summands in $\cT$. By construction all the morphisms are non-zero. Thus, either the morphisms eventually become isomorphisms sufficiently far to the right, in which case we are done, or there is an arbitrarily long sequence of non-trivial morphisms. However, since each  $\refinedY{t}{i+kd}' \in \sY_i \cap \cT$ is made up of summands of bounded length, the composition of sufficiently many non-trivial morphisms is zero, a contradiction.

Now suppose $\card{\ind (\sY_i \cap \cT)} = \infty$ for all $i \in \{0,\ldots,d-1\}$ and consider the sequence \eqref{eq:main-sequence} of summands in $\cT$. In this case, each  $\refinedY{t}{n}'$ is made up of summands of length at most $\tuberk \ceil{\len{t}/\tuberk}$ by Proposition~\ref{prop:length-bound}, and the result follows using the same argument.
\end{proof}

\subsection{Proof of Theorem~\ref{thm:small:extended}}\label{sec:proof-small-extended}

\begin{proof}
$(b) \implies (a)$: Follows immediately from Theorem~\ref{thm:x_n}.

$(a) \implies (c)$: Suppose that $(\sX^I,\sY^I)$ is a t-structure. Let $\cN$ be a connected non-regular component of $\sT$ such that $\sY_i\cap\cN$ has an infinite strictly increasing sequence for all $i\in I$. We prove that for any indecomposable object $t \in \cN$ there is an object in $ \sY^I \cap \cN$ which is greater or equal to $t$ with respect to the partial order on $\cN$: this object is realised as a summand of the left truncation of $t$ with respect to the t-structure $(\sX^I,\sY^I)$. It immediately follows that there exists an infinite strictly increasing sequence in $ \sY^I \cap \cN$.

Consider the truncation triangle
$\tri{t_{\sX^I}}{t}{t_{\sY^I}}$
with respect to $(\sX^I,\sY^I)$. By Lemma~\ref{lem:nonregint}, $t_{\sX^I} \notin \cN$. Note that $t_{\sX^I}$ cannot have any summands lying in $\Sigma^{-1} \cN$, since $\sX^I$ is closed under suspension, and therefore if $t_{\sX^I}$ has summands lying in $\Sigma^{-1} \cN$, then $\sX^I \cap \cN$ is non-zero; a contradiction.  Thus, $x^I$ has summands lying in $\cR$, the unique regular components admitting morphisms to objects in $\cN$.
It follows that $t_{\sY^I}$ has a non-regular summand, since the full subcategory of regular objects is a triangulated subcategory of $\sT$, thus if $t_{\sY^I}$ were regular, then $t$ would also be regular; a contradiction. Now since $t$ admits non-trivial morphisms to each summand  to $t_{\sY^I}$, it follows that the non-regular summands of $t_{\sY^I}$ are strictly greater than $t$ with respect to the partial order on $\cN$.
This gives the desired conclusion.

$(c) \implies (b):$ Let $t \in \ind(\sT)$. Without loss of generality we may assume that it lies in the heart $\sH$ of some standard t-structure on $\sT$. Suppose that the sequence $(\refinedY{t}{n})$ does not terminate. Then there exists a subsequence 
 $ t \to t_0 \to t_1 \to t_2 \to \cdots $
of indecomposable objects such that all compositions of morphisms are non-zero and which does not terminate. 

\noindent\emph{Claim.} \emph{Only a finite piece of this sequence lies in any component $\cC$ of the AR quiver.}

For each integer $k \equiv -1 \bmod d$, we know from Lemma~\ref{lem:summands} that $t_{k+n}$ is a summand of $\refinedY{t_k}{n}$. If $t_k \in \cC$ then there exists $n_0 \in \IN$ such that for $n \geq n_0$ any summand of $\refinedY{t_k}{n}$ contained in $\cC$ lies in $\sY^I$; this is Lemma~\ref{lem:nonregulartermination} for $\cC$ is non-regular, and Proposition~\ref{prop:regulartermination} for a regular component $\cC$. Since elements in $\sY^I$ are fixed by the algorithm and by construction $(t_n)$ does not terminate, this implies $t_{k+n}\notin\cC$ for $n \geq n_0$; giving the claim.

There is an ordering of the components of $\sT$ with morphisms only going in one direction. Note that if $t_i\in\cR$ and $t_j\in\cR'$ for regular components with $j > i$, then $\cR'=\cR$ or $\cR'=\Sigma \cR$, because there are no non-zero morphisms between different regular components. Therefore, there exists $m_0 \in \IN$ such that $t_m \notin \sH \cup \Sigma \sH$ for $m \geq m_0$. Since $\sT$ is piecewise hereditary, we get $\Hom(t,t_m)=0$, contradicting our assumption about about non-zero compositions of morphisms in the sequence $(t_n)$. Hence, the sequence must terminate.
\end{proof}


\section{Group actions and equivariant categories} \label{sec:equivariant}

\noindent
Let $\sT$ be a $\kk$-linear category for an algebraically closed field $\kk$ of characteristic 0. We are interested in triangulated categories but the basic definitions work in greater generality.

Let $G$ be a finite group. An \emph{action of $G$ on $\sT$} is given by a collection of functors $g\colon\sT\isom\sT$, one for each group element (by abuse of notation the functor is denoted by the group element) such that the identity element is assigned the identity functor and there is a given functor isomorphism for each relation in $G$, satisfying the natural compatibility condition, see \cite{Deligne}. We point out that this does not just entail a group homomorphism $G\to\Aut(\sT)$. Rather, the action is given by a functor of groupoids $\underline G\to\underline{\Aut}(\sT)$.

Typical examples are given by the action of $G$ on an $\kk$-algebra $\Lambda$, in which case $G$ also acts on $\mod{\Lambda}$ and thus also on $\Db(\Lambda)$; similarly, an action of $G$ on a variety $X$ induces $G$-actions on $\coh(X)$ and $\Db(X)$ --- however here we are concerned with group actions which are (at least at first glance) not of this obvious type.

\subsection*{Linearised objects and the equivariant category}

We form a new category $\sT^G$ whose objects are $G$-objects, i.e.\ pairs $(t,\tau)$ where $t\in\sT$ is an object and $\tau_g\colon t\isom g(t)$ a collection of isomorphisms (called \emph{linearisations}) with the obvious compatibility relations. Given two pairs $(t,\tau)$ and $(t',\tau')$, the vector space $\Hom_\sT(t,t')$ becomes a $G$-representation in a natural way, using the linearisations and we set $\Hom_{\sT^G}((t,\tau),(t',\tau')):=\Hom_\sT(t,t')^G$.

It is a standard fact that $\sT^G$ is a category, which we call the \emph{$G$-equivariant category}. If $\sT$ is an abelian category, then so is $\sT^G$. If $\sT$ is an \emph{algebraic} triangulated category, i.e.\ is presented as the homotopy category of a dg category, then $\sT^G$ is triangulated as well \cite{Sosna}. An additional condition like this is needed because of the non-functoriality of cones in abstract triangulated categories.

\begin{example}
A $G$-action on a smooth variety $X$ induces a canonical equivalence of abelian categories, $\coh(X)^G=\coh([X/G])$ where $[X/G]$ is the quotient \emph{stack}. It is a smooth Deligne-Mumford stack, whereas the quotient variety $X/G$ has singularities at fixed points. For example, almost all `weighted projective lines' occur as quotient stacks $[C/G]$ where $C$ is a smooth projective curve and $G\subset\Aut(C)$ such that 
$C/G\cong\IP^1$.
\end{example}

\begin{remark} There is another natural categorical construction: the \emph{orbit category} $\sT/G$ has the same objects as $\sT$ and morphism spaces $\Hom_{\sT/G}(t,t'):=\bigoplus_{g\in G}\Hom_\sT(t,g(t'))$. In $\sT/G$, objects $t$ and $g(t)$ become isomorphic. 
If $\sT$ is a triangulated category with infinite direct sums, this can be extended to countable groups; for example, $\sT/\IZ[1]$ is known as the `orbit category' of $\sT$.
In contrast with equivariant categories, $\sT$ abelian or triangulated does in general \emph{not} descend to orbit categories.
\end{remark}

\subsection*{Application}

A t-structure $(\sX,\sY)$ is \emph{$G$-invariant} if $g\sX=\sX$ and $g\sY=\sY$ for all $g\in G$.

\begin{proposition} \label{prop:equivariant-t-structure}
Let $(\sX,\sY)$ be a $G$-invariant t-structure on $\sT$. Then the pair $(\sX^G,\sY^G)$ of $G$-equivariant categories forms a t-structure on $\sT^G$.
\end{proposition}

\begin{proof}
Note that the group action on $\sT$ induces actions on $\sX$ and $\sY$. The equivariant categories $\sX^G$ and $\sY^G$ are full, additive subcategories of $\sT^G$. (Note that here $\sX^G$ and $\sY^G$ are subcategories of $\sT^G$, not of $\sT$.) We proceed to check the axioms for a t-structure:
Orthogonality $\Hom_{\sT^G}(\sX^G,\sY^G)=0$ follows immediately from $\Hom_\sT(\sX,\sY)=0$. The same is true for invariance under (de)suspension.

As to the truncation triangles, let $(t,\tau)\in\sT^G$. There is a truncation triangle $x\to t\to y$ in $\sT$ with respect to $(\sX,\sY)$. Applying an autoequivalence $g$ yields another truncation triangle $g(x)\to g(t)\to g(y)$, with respect to $(g\sX,g\sY)$, i.e.\ again for $(\sX,\sY)$, as the t-structure is $G$-invariant. Furthermore, we have $\tau_g\colon t\isom g(t)$, and therefore both triangles are truncations of $t\cong g(t)$ with respect to $(\sX,\sY)$. Uniqueness of truncations enforces $x\cong g(x)$ and $y\cong g(y)$. More precisely, the truncation $x\to t\xxrightarrow{\tau} g(t)$ and the morphism $g(x)\to g(t)$ induce $x\to g(x)$, which then must be an isomorphism. In this way, we get canonical isomorphisms $\xi_g\colon x\isom g(x)$ for all $g\in G$ (and similarly $\mu_g\colon y\isom g(y)$). This equips $x$ and $y$ with $G$-linearisations, and the triangle $(x,\xi)\to(t,\tau)\to(y,\mu)$ is exact in $\sT^G$ and then a truncation triangle with respect to $(\sX^G,\sY^G)$.
\end{proof}

\begin{example}
Let $\sT:=\Db(\kk A_5)$ be the bounded derived category of the path algebra of the $A_5$ quiver. Consider the t-structure pictured in Figure~\ref{fig:A5_t-structure}. Furthermore, consider the autoequivalence $g$ induced by the horizontal reflection of the AR quiver of $\sT$, so that $G:=\langle g\rangle \cong \IZ/2\IZ$ acts on $\sT$. As this is the representation finite case, averaging is possible by Proposition~\ref{prop:repfin}, and Figure~\ref{fig:A5_t-structure} also shows the invariant t-structure on $\sT$. Finally, this invariant t-structure on $\sT$ induces a t-structure on $\sT^G$. It is well-known that $\sT^G\cong\Db(\kk D_4)$, and the t-structure is presented in Figure~\ref{fig:D4_induced}, as well.
\end{example}

\begin{figure}
\centering
\includegraphics[width=1.0\textwidth]{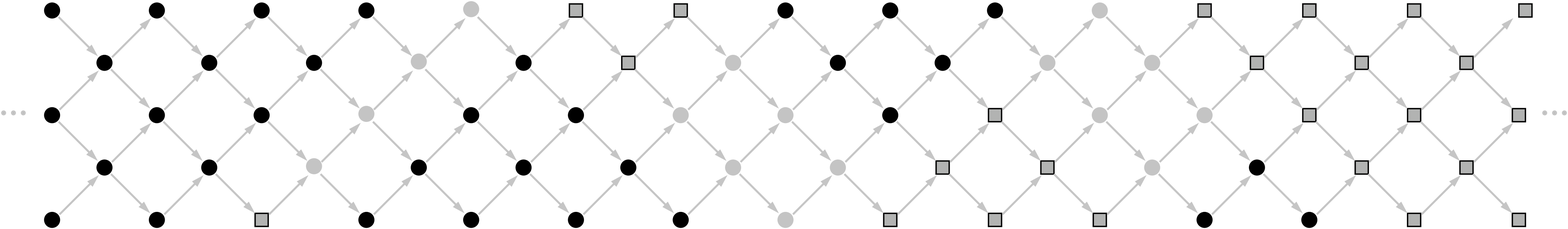}
\includegraphics[width=1.0\textwidth]{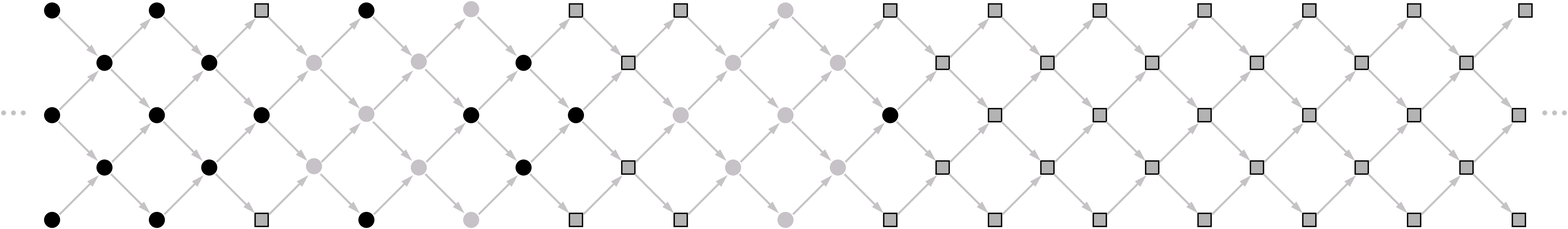}   

\bigskip
\includegraphics[width=1.0\textwidth]{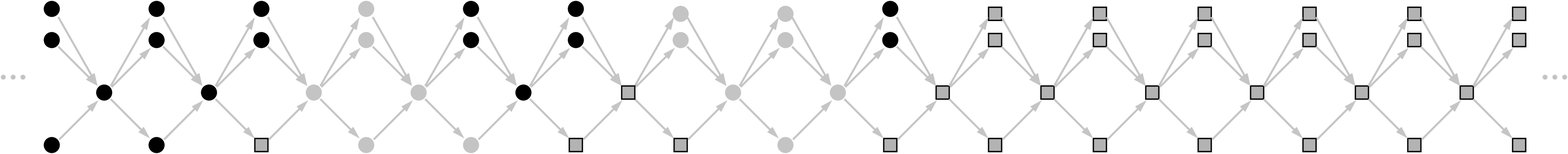}
\caption[A t-structure on $\sT=\Db(\kk
  A_5))$.]{ \label{fig:A5_t-structure} \label{fig:D4_induced} \small 
   Top: A t-structure on $\sT=\Db(\kk A_5)$ given by $\sX$ (\dotspace\boxeddot) and $\sY$ (\dotspace\blackdot).
   Middle: The averaged t-structure for the $\IZ/2\IZ$-action on $\sT$. \newline
   Bottom: The induced t-structure $(\sX^G,\sY^G)$ on $\sT^G\cong\Db(\kk D_4)$.
}
\end{figure}

\subsection*{Preservation of properties for averaged t-structures} We close with a comment about properties of t-structures and their preservation under averaging. A t-structure $(\sX,\sY)$ in $\sT$ is called
\begin{enumerate}
\item \emph{stable} if $\Sigma \sX = \sX$ (and hence $\Sigma \sY = \sY$),
\item \emph{bounded} if $\bigcup_{i\in\IZ} \Sigma^i \sX = \bigcup_{i\in\IZ} \Sigma^i \sY = \sT$,
\item \emph{non-degenerate} if $\bigcap_{i\in\IZ} \Sigma^i \sX = \bigcap_{i\in\IZ} \Sigma^i \sY = \{0\}$.
\end{enumerate}
A stable t-structure $(\sX,\sY)$ is often called a `semi-orthogonal decomposition' of $\sT$ into triangulated subcategories $\sY$, $\sX$, and denoted $\langle\sY,\sX\rangle$. The following statement follows from unravelling the definitions.

\begin{proposition}
Let $(\sX_i,\sY_i)_{i\in I}$ be a finite set of t-structure on a triangulated category $\sT$ and assume that $(\sX^I,\sY^I)$ is a t-structure. If $(\sX_i,\sY_i)$ is stable (or bounded) for all $i\in I$, then so is $(\sX^I,\sY^I)$.
\end{proposition}

\begin{corollary}
Let $G$ be a finite group acting on an algebraic triangulated category $\sT$. Assume that $(\sX,\sY)$ is a t-structure on $\sT$ such that the average $(\sX^G,\sY^G)$ is again a t-structure on $\sT$. If $(\sX,\sY)$ is stable (or non-degenerate, or bounded), then the same is true for the induced equivariant t-structure on $\sT^G$.
\end{corollary}

In contrast, it is unclear whether non-degeneracy is preserved. This problem boils down to the following question: can the extension closure of two aisles contain a triangulated subcategory, assuming that neither of the aisles contains triangulated subcategories?


\addtocontents{toc}{\protect{\setcounter{tocdepth}{-1}}}

\bigskip
\noindent
\begin{tabular}{@{} l p{0.5\textwidth}}
Email addresses: & \verb+broomhead@math.uni-hannover.de+ 
                   \verb+pauk@math.uni-hannover.de+ 
                   \verb+ploog@math.uni-hannover.de+
\end{tabular}

\medskip \noindent
Web page of David Pauksztello: \verb+http://www.iazd.uni-hannover.de/~pauksztello+

\end{document}